\documentclass[10pt]{amsart}

\usepackage{amsmath}
\usepackage{amsthm}
\usepackage{amssymb}

\usepackage[arrow, curve, frame, matrix, graph, tips]{xy}
\usepackage{diagrams}

\theoremstyle{plain}
\newtheorem{dfn}[subsection]{Definition}
\newtheorem{thm}[subsection]{Theorem}
\newtheorem{prp}[subsection]{Proposition}
\newtheorem{cor}[subsection]{Corollary}
\newtheorem{lma}[subsection]{Lemma}

\theoremstyle{remark}
\newtheorem{rmk}[subsection]{Remark}

\def\Dd{\mathcal{D}}
\def\Ee{\mathcal{E}}
\def\Ff{\mathcal{F}}

\def\eps{\epsilon}

\def\EssIm{\mathrm{EssIm}}
\def\Aa{\mathcal{A}}
\def\Bb{\mathcal{B}}
\def\inc{\hookrightarrow}
\def\op{\mathrm{op}}
\def\Mod{\mathrm{Mod}}
\def\Mnd{\mathrm{Mnd}}
\def\Th{\mathrm{Th}}

\def\CC{\mathbb{C}}
\def\DD{\mathbb{D}}
\def\Sh{\mathrm{Sh}}
\def\Acyc{\mathit{Acyc}}

\def\Sets{\mathrm{Sets}}
\def\Ff{\mathcal{F}}
\def\com{\mathrm{com}}
\def\NN{\mathbb{N}}
\def\gen{\mathrm{gen}}
\def\sym{\mathrm{sym}}

\newcommand{\tn}[1]{\textnormal{#1}}
\newcommand{\tnb}[1]{\textnormal{\bf #1}}

\newcommand{\PSh}[1]{\widehat{#1}}
\newcommand{\Z}{\mathbb{Z}}
\newcommand{\N}{\mathbb{N}}

\newcommand{\id}{\tn{id}}
\newcommand{\ca}[1]{\mathcal{#1}}
\newcommand{\ladj}{\dashv}
\newcommand{\iso}{\cong}
\newcommand{\catequiv}{\simeq}
\newcommand{\Graph}{\tnb{Gph}}
\newcommand{\GraphSite}{\mathbb{G}_{{\leq}1}}
\newcommand{\InvGraph}{{\tn{i-}}{\tnb{Gph}}}
\newcommand{\InvGraphSite}{{\tn{i-}}\mathbb{G}_{{\leq}1}}
\newcommand{\Gpd}{\tnb{Gpd}}
\newcommand{\Fact}[3]{\tn{Fact}_{#1,#2}(#3)}
\DeclareMathOperator*{\colim}{\tn{colim}}

\begin{document}
\title{Monads with arities and their associated theories}

\author{Clemens Berger, Paul-Andr\'e Melli\`es and Mark Weber}

\date{February 10, 2011}

\subjclass{Primary 18C10, 18C15; Secondary 18D50, 18B40}
\keywords{Algebraic theory, density, arity, monad, operad, groupoid}

\begin{abstract}After a review of the concept of ``monad with arities'' we show that the category of algebras for such a monad has a canonical dense generator. This is used to extend the correspondence between finitary monads on sets and Lawvere's algebraic theories to a general correspondence between monads and theories for a given category with arities. As application we determine arities for the free groupoid monad on involutive graphs and recover the symmetric simplicial nerve characterisation of groupoids.\end{abstract}

\maketitle

\section*{Introduction.}
In his seminal work \cite{La} Lawvere constructed for every variety of algebras, defined by finitary operations and relations on sets, an algebraic theory whose $n$-ary operations are the elements of the free algebra on $n$ elements. He showed that the variety of algebras is equivalent to the category of models of the associated algebraic theory. A little later, Eilenberg-Moore \cite{EM} defined algebras for monads (``triples''), and it became clear that Lawvere's construction is part of an equivalence between (the categories of) finitary monads on sets and algebraic theories respectively; for an excellent historical survey we refer the reader to Hyland-Power \cite{HP}, cf. also the recent book by Ad\'amek-Rosick\'y-Vitale \cite{ARV}.

In \cite{Be} the first named author established a formally similar equivalence between certain monads on globular sets (induced by Batanin's globular operads \cite{Ba}) and certain globular theories, yet he did not pursue the analogy with algebraic theories any further. The main purpose of this article is to develop a framework in which such \emph{monad/theory correspondences} naturally arise. Central to our approach is the so-called \emph{nerve theorem} which gives sufficient conditions under which algebras over a monad can be represented as models of an appropriate theory.


For the general formulation we start with a dense generator $\Aa$ of an arbitrary category $\,\Ee$. The objects of $\Aa$ are called the \emph{arities} of $\Ee$. According to the third named author \cite{W} (on a suggestion of Steve Lack) a monad $T$ on $\Ee$ which preserves the density presentation of the arities in a strong sense (cf. Definition \ref{monadwitharities}) is called a \emph{monad with arities} $\Aa$. The associated theory $\Theta_T$ is the full subcategory of the Eilenberg-Moore category $\Ee^T$ spanned by the free $T$-algebras on the arities. The nerve theorem identifies then $T$-algebras with $\Theta_T$-models, i.e. presheaves on $\Theta_T$ which typically take certain colimits in $\Theta_T$ to limits in sets (cf. Definition \ref{dfn:theory}).

The algebraic theories of Lawvere arise by taking $\Ee$ to be the category of sets, and $\Aa$ (a skeleton of) the full subcategory of finite sets (cf. Section \ref{sct:algebraictheory}). Our terminology is also motivated by another example; namely, if $\Ee$ is the category of directed graphs, $\Aa$ the full subcategory spanned by finite directed edge-paths, and $T$ the free category monad, the associated theory is the simplex category $\Delta$, and the nerve theorem identifies small categories with simplicial sets fulfilling certain exactness conditions, originally spelled out by Grothendieck \cite{G} and Segal \cite{Se0}. More generally, if $\Ee$ is the category of globular sets, $\Aa$ the full subcategory of globular pasting diagrams, and $T$ the free $\omega$-category monad, the associated theory is Joyal's cell category \cite{J}, and the nerve theorem identifies small $\omega$-categories with cellular sets fulfilling generalised Grothendieck-Segal conditions \cite{Be}.

Inspired by these examples and by Leinster's nerve theorem \cite{L} for strongly cartesian monads on presheaf categories, the third named author \cite{W} established a general nerve theorem for monads with arities on cocomplete categories. Recently the second named author \cite{M} observed that there is no need of assuming cocompleteness and that the concepts of theory and model thereof carry over to this more general context. He sketched a $2$-categorical proof of a general monad/theory correspondence on the basis of Street-Walter's \cite{SW} axiomatics for Yoneda structures.

The following text contains concise proofs of the nerve theorem and the resulting monad/theory correspondence. The flexibility of our approach lies in the relative freedom for the choice of convenient arities: their density is the only requirement. Different choices lead to different classes of monads and to different types of theories. The \emph{rank} of a monad is an example of one possible such choice. We have been careful to keep the formalism general enough so as to recover the known examples. We have also taken this opportunity to give a unified account of several key results of \cite{Be,W0,W,M}, which hopefully is useful, even for readers who are familiar with our individual work. Special attention is paid to the free groupoid monad on involutive graphs for reasons explained below. The article is subdivided into four sections:\vspace{1ex}

Section 1 gives a new and short proof of the nerve theorem (cf. Theorem \ref{T2}) based on the essential image-factorisation of strong monad morphisms. We show that classical results of Gabriel-Ulmer \cite{GU} and Ad\'amek-Rosick\'y \cite{AR} concerning Eilenberg-Moore categories of $\alpha$-accessible monads in $\alpha$-accessible categories can be considered as corollaries of our nerve theorem (cf. Theorem \ref{T3}).

Section 2 is devoted to alternative formulations of the concept of monad with arities. We show in Proposition \ref{prp:EM-object} that  monads with arities are precisely the monads (in the sense of Street \cite{St0}) of the $2$-category of categories with arities, arity-respecting functors and natural transformations. In Proposition \ref{prp:monad-with-arities-elementary} arity-respecting functors are characterised via the connectedness of certain factorisation categories. We study in some detail \emph{strongly cartesian} monads, i.e. cartesian monads which are local right adjoints, and recall from \cite{W0,W} that they allow a calculus of \emph{generic factorisations}. This is used in Theorem \ref{thm:lra->arities} to show that every strongly cartesian monad $T$ comes equipped with \emph{canonical arities} $\Aa_T$. The shape of these canonical arities is essential for the behaviour of the associated class of monads. The monads induced by \emph{$T$-operads} in the sense of Leinster \cite{L0} are monads with arities $\Aa_T$.

Section 3 introduces the concept of theory appropriate to our level of generality, following \cite{M}. The promised equivalence between monads and theories for a fixed category with arities is established in Theorem \ref{thm:monad-theory}. This yields as a special case the correspondence between finitary monads on sets and Lawvere's algebraic theories. We introduce the general concept of a \emph{homogeneous theory}, and obtain in Theorem \ref{thm:cartesian-homogeneous}, for each strongly cartesian monad $T$ (whose arities have no symmetries), a correspondence between $T$-operads and $\Theta_T$-homogeneous theories. This yields in particular the correspondence \cite{Be,A} between Batanin's globular $\omega$-operads and $\Theta_\omega$-homogeneous theories where $\Theta_\omega$ denotes Joyal's cell category. In Section \ref{sct:Gamma} we show that \emph{symmetric operads} can be considered as $\Gamma$-homogeneous theories, where the category $\Gamma$ of Segal \cite{Se} is directly linked with the algebraic theory of commutative monoids. This is related to recent work by Lurie \cite{Lu} and Batanin \cite{Ba2}.

Section 4 studies the \emph{free groupoid monad} on the category of involutive graphs. This example lies qualitatively in between the two classes of monads with arities which have been discussed so far, namely the $\alpha$-accessible monads on $\alpha$-accessible categories (with arities the $\alpha$-presentable objects) and the strongly cartesian monads on presheaf categories (endowed with their own canonical arities). Indeed, the category of involutive graphs is a presheaf category, but the free groupoid monad is not cartesian (though finitary). In Theorem \ref{thm:A-arities-for-G} we show that the finite connected acyclic graphs (viewed as involutive graphs) endow the free groupoid monad with arities, and that this property may be used to recover Grothendieck's symmetric simplicial characterisation of groupoids \cite{G} as an instance of the nerve theorem.

Let us briefly mention some further developments and potential applications.

-- Our methods should be applicable in an enriched setting, in the spirit of what has been done for algebraic theories by Nishizawa-Power \cite{NP}. Ideally, the $2$-category of categories with arities (cf. \ref{section:arity-respecting}) could be replaced with an enriched version of it.

-- The monad/theory correspondence of Section 3 strongly suggests a combinatorial formulation of \emph{Morita equivalence} between monads with same arities.  Such a concept would induce a theory/variety duality as the one established by Ad\'amek-Lawvere-Rosick\'y \cite{ALR} for idempotent-complete algebraic theories.

-- Our notion of \emph{homogeneous theory} captures the notion of \emph{operad} in two significant cases: globular operads (cf. \ref{sct:Theta}) and symmetric operads (cf. \ref{sct:Gamma}). The underlying conceptual mechanism needs still to be clarified.

-- A future extension of our framework will contain a formalism of \emph{change-of-arity} functors. A most interesting example is provided by the symmetrisation functors of Batanin \cite{Ba2} which convert globular $n$-operads into symmetric operads (cf. \ref{sct:Gamma}).

-- The treatment in Section 4 of the free groupoid monad on involutive graphs is likely to extend in a natural way to the free $n$-groupoid monad on involutive $n$-globular sets. This is closely related to recent work by Ara \cite{A,A2}.

-- The notion of monad with arities sheds light on the concept of \emph{side effects} in programming languages. It should provide the proper algebraic foundation for a presentation of \emph{local stores} in an appropriate presheaf category, cf. \cite{HP,M}.\vspace{1ex}

\emph{Acknowledgements:} We thank the organisers of the Category Theory Conference CT2010 in Genova, especially Giuseppe Rosolini, for the stimulating atmosphere of this conference, which has been at the origin of this article. We are grateful to Jiri Ad\'amek, Dimitri Ara, Michael Batanin, Brian Day, Martin Hyland, Steve Lack, Georges Maltsiniotis and Eugenio Moggi for helpful remarks and instructive discussions.\vspace{1ex}

\emph{Notation and terminology:} All categories are supposed to be locally small. For a monad $T$ on a category $\Ee$, the \emph{Eilenberg-Moore} and \emph{Kleisli} categories of $T$ are denoted $\Ee^T$ and $\Ee_T$ respectively. An isomorphism-reflecting functor is called \emph{conservative}. The category of set-valued presheaves on $\Aa$ is denoted $\PSh\Aa$. For a functor $j:\Aa\to\Bb$, the left and right adjoints to the restriction functor $j^*:\PSh\Bb\to\PSh\Aa$ are denoted $j_!$ and $j_*$ respectively, and called \emph{left} and \emph{right Kan extension} along $j$.

\section{The nerve theorem.}

\subsection{Monad morphisms}\label{dfn:monad-morphism}--\vspace{1ex}

For a monad $(T,\mu,\eta)$ on a category $\Ee$, the Eilenberg-Moore category $\Ee^T$ of~$T$ comes equipped with forgetful functor $U:\Ee^T\to\Ee$ and left adjoint $F:\Ee\to\Ee^T$. The objects of~$\Ee^T$ are the \emph{$T$-algebras} in~$\Ee$, i.e. pairs $(X,\xi_X)$ consisting of an object $X$ and a morphism $\xi_X:TX\to X$ such that $\xi_X\cdot T\xi_X=\xi_X\mu_X$ and $\xi_X\eta_X=id_X$.

Following Street \cite{St0}, for given categories $\Ee_1$ and $\Ee_2$ with monads $(T_1,\mu_1,\eta_1)$ and $(T_2,\mu_2,\eta_2)$ respectively, a \emph{monad morphism} $(\Ee_1,T_1)\to(\Ee_2,T_2)$ is defined to be a pair $(\Phi,\rho)$ consisting of a functor $\Phi:\Ee_1\to\Ee_2$ and a natural transformation $\rho:T_2\Phi\Rightarrow\Phi T_1$ such that the following two diagrams commute:\vspace{1ex}

\begin{gather}\begin{diagram}[small,silent]\label{monad-identities}T_2T_2\Phi&\rImplies^{T_2\rho}&T_2\Phi T_1&\rImplies^{\rho T_1}&\Phi T_1T_1&\quad&\quad&\Phi&\rImplies^{\Phi\eta_1}&\Phi T_1\\\dImplies^{\mu_2\Phi}&&&&\dImplies_{\Phi\mu_1}&\quad&\quad&\dImplies^{\eta_2\Phi}&\ruImplies&\!\!\!\!\rho\\T_2\Phi&&\rImplies^{\rho}&&\Phi T_1&\quad&\quad&T_2\Phi&&\end{diagram}\end{gather}

A monad morphism $(\Phi,\rho):(\Ee_1,T_1)\to(\Ee_2,T_2)$ induces a commutative square \begin{gather}\label{mon}\begin{diagram}[small]\Ee_1^{T_1}&\rTo^{\overline{\Phi}}&\Ee_2^{T_2}\\\dTo^{U_1}&&\dTo_{U_2}\\\Ee_1&\rTo^{\Phi}&\Ee_2\end{diagram}\end{gather}
in which the functor $\overline{\Phi}:\Ee_1^{T_1}\to\Ee_2^{T_2}$ takes the $T_1$-algebra $(X,\xi_X)$ to the $T_2$-algebra $(\Phi X,\Phi\xi_X\cdot\rho_X)$. Conversely, any commutative square (\ref{mon}) induces a transformation $\rho=U_2\eps_2\overline{\Phi}F_1\circ T_2\Phi\eta_1:T_2\Phi\Rightarrow\Phi T_1$ fulfilling the identities (\ref{monad-identities}). It is then straightforward to check that this establishes a one-to-one correspondence between monad morphisms $(\Phi,\rho):(\Ee_1,T_1)\to(\Ee_2,T_2)$ and liftings $\overline{\Phi}:\Ee_1^{T_1}\to\Ee_2^{T_2}$ like in (\ref{mon}).

\subsection{Isofibrations and essential image-factorisations}\label{replete}--\vspace{1ex}

Recall that any functor $\Phi:\Ee_1\to\Ee_2$ factors as an essentially surjective functor $\Ee_1\to\EssIm(\Phi)$ followed by the inclusion $\EssIm(\Phi)\to\Ee_2$ of a full and replete subcategory. By definition, the \emph{essential image} $\EssIm(\Phi)$ is the full subcategory of $\Ee_2$ spanned by those objects which are isomorphic to an object in the image of $\Phi$. We call this factorisation the \emph{essential image-factorisation} of $\Phi$.

Observe that a subcategory is replete precisely when the inclusion is an isofibration. We use the term \emph{isofibration} for those functors $F:\Dd\to\Ee$ which have the property that for any isomorphism in $\Ee$ of the form $g:X\cong F(Y)$ there exists an isomorphism in $\Dd$ of the form $f:X'\cong Y$ such that $F(f)=g$. Isofibrations between small categories form the class of fibrations for the Joyal-Tierney model structure \cite{JT} on the category of small categories. The essential image-factorisation is the (up to isomorphism) unique factorisation $\Phi=\Phi_2\Phi_1$ such that $\Phi_1$ is essentially surjective and $\Phi_2$ is injective on objects, fully faithful and an isofibration.

\begin{prp}\label{P1}For any monad morphism $(\Phi,\rho):(\Ee_1,T_1)\to(\Ee_2,T_2)$ with induced lifting $\overline{\Phi}:\Ee_1^{T_1}\to\Ee_2^{T_2}$, the following properties hold:

\begin{itemize}\item[(a)]If $\Phi$ is faithful, then so is $\overline{\Phi}$.\item[(b)]If $\Phi$ is fully faithful and $\rho$ pointwise epimorphic, then $\overline{\Phi}$ is fully faihtful.\item[(c)]If $\Phi$ is fully faithful and $\rho$ an isomorphism, then $\overline{\Phi}$ is the pullback of~$\Phi$ along $U_2$; moreover, the essential image-factorisation of~$\overline{\Phi}$ may be identified with the pullback along $U_2$ of the essential image-factorisation of~$\Phi$.\end{itemize}\end{prp}

\begin{proof}--

(a) The forgetful functor $U_1:\Ee_1^{T_1}\to\Ee_1$ is faithful. Therefore, if $\Phi$ is faithful so is $\Phi U_1=U_2\overline{\Phi}$, and hence $\overline{\Phi}$ is faithful as well.

(b) It remains to be shown that $\overline{\Phi}:\Ee_1^{T_1}\to\Ee_2^{T_2}$ is full. Let $(X,\xi_X)$ and $(Y,\xi_Y)$ be $T_1$-algebras and $g:\overline{\Phi}(X,\xi_X)\to\overline{\Phi}(Y,\xi_Y)$ be a map of $T_2$-algebras. By definition, we have $\overline{\Phi}(X,\xi_X)=(\Phi X,\Phi\xi_X\cdot\rho_X)$ and $\overline{\Phi}(Y,\xi_Y)=(\Phi Y,\Phi\xi_Y\cdot\rho_Y)$. Since $\Phi$ is full, there is map $f:X\to Y$ in $\Ee_1$ such that $\Phi f=g$.

We will show that $f$ is actually a map of $T_1$-algebras such that $\overline{\Phi}f=g$. Indeed, in the following diagram\begin{diagram}[small]T_2\Phi X&\rTo^{T_2\Phi f}&T_2\Phi Y\\\dTo^{\rho_X}&&\dTo_{\rho_Y}\\\Phi T_1 X&\rTo^{\Phi T_1 f}&\Phi T_1 Y\\\dTo^{\Phi \xi_X}&&\dTo_{\Phi \xi_Y}\\\Phi X&\rTo^{\Phi f}&\Phi Y\end{diagram}the outer rectangle commutes since $g=\Phi f$ is a map of $T_2$-algebras, the upper square commutes by naturality of $\rho$, so that the lower square also commutes because $\rho_X$ is an epimorphism. Therefore, since $\Phi$ is faithful, $f$ is a map of $T_1$-algebras.

(c) Let $\Psi:\Ee_1\times_{\Ee_2}\Ee_2^{T_2}\to\Ee_2^{T_2}$ be the categorical pullback of $\Phi$ along $U_2$. This functor is fully faithful since $\Phi$ is. Moreover,  $\overline{\Phi}=\Psi Q$ for a unique functor $Q:\Ee_1^{T_1}\to\Ee_1\times_{\Ee_2}\Ee_2^{T_2}$. Since by (b) $\overline{\Phi}$ is fully faithful, $Q$ is fully faithful as well. It remains to be shown that $Q$ is bijective on objects, i.e. that for each object $X$ of $\Ee_1$, $T_2$-algebra structures $\xi_{\Phi(X)}:T_2\Phi X\to\Phi X$ are in one-to-one correspondence with $T_1$-algebra structures $\xi_X:T_1 X\to X$ such that $\xi_{\Phi X}=\rho_X\Phi\xi_X$. Since $\rho_X$ is invertible and $\Phi$ fully faithful, we must have $\xi_X=\Phi^{-1}(\rho_X^{-1}\xi_{\Phi(X)})$. It is easy to check that this defines indeed a $T_1$-algebra structure on $X$.

For the second assertion, observe that the monad $(T_2,\mu_2,\eta_2)$ restricts to a monad $(T,\mu,\eta)$ on the essential image of $\Phi$, since by hypothesis $\rho:T_2\Phi\Rightarrow\Phi T_1$ is invertible. We get thus a monad morphism $(\Phi_1,\rho_1):(\Ee_1,T_1)\to(\EssIm(\Phi),T)$ by corestriction, as well as a monad morphism $(\Phi_2,\rho_2):(\EssIm(\Phi),T)\to(\Ee_2,T_2)$ with $\rho_2$ being an identity $2$-cell. Since by construction $(\Phi,\rho)$ is the composite monad morphism $(\Phi_2,\rho_2)(\Phi_1,\rho_1)$ we have the following commutative diagram\begin{diagram}[small]\Ee_1^{T_1}&\rTo^{\overline{\Phi}_1}&\EssIm(\Phi)^T&\rInto^{\overline{\Phi}_2}&\Ee_2^{T_2}\\\dTo^{U_1}&&\dTo^U&&\dTo^{U_2}\\\Ee_1&\rTo^{\Phi_1}&\EssIm(\Phi)&\rInto^{\Phi_2}&\Ee_2\end{diagram}in which the Eilenberg-Moore category $\EssIm(\Phi)^T$ may be identified with the categorical pullback $\EssIm(\Phi)\times_{\Ee_2}\Ee_2^{T_2}$. The left square is also a pullback since $\Phi_1$ is fully faithful and $\rho_1$ is invertible. All vertical functors are isofibrations. Therefore, since the pullback of an essentially surjective functor along an isofibration is again essentially surjective, the lifting $\overline{\Phi}_1$ of $\Phi_1$ is essentially surjective and fully faithful. The lifting $\overline{\Phi}_2$ of $\Phi_2$ is injective on objects. Moreover, since $U$ and $\Phi_2$ are isofibrations, the composite functor $\Phi_2 U=U_2\overline{\Phi}_2$ is an isofibration as well. Finally, since $U_2$ is a faithful isofibration, $\overline{\Phi}_2$ is itself an isofibration. Therefore, the factorisation $\overline{\Phi}=\overline{\Phi}_2\overline{\Phi}_1$ can be identified with the essential image-factorisation of $\overline{\Phi}$, and arises from the essential image-factorisation of $\Phi$ by pullback along $U_2$.\end{proof}

\subsection{Exact adjoint squares}\label{sct:adjoint}--\vspace{1ex}

We now want to extend Proposition \ref{P1} to more general squares than those induced by monad morphisms. To this end recall that a functor $R:\Dd\to\Ee$ is \emph{monadic} if $R$ admits a left adjoint $L:\Ee\to\Dd$ such that the comparison functor $K:\Dd\to\Ee^{RL}$ is an equivalence of categories. As usual, the comparison functor $K$ takes an object $Y$ of $\Dd$ to the $RL$-algebra $(RY,R\eps_Y)$, where $\eps$ denotes the counit of the adjunction and $(RL,R\eps L,\eta)$ is the monad induced by the adjunction, cf. Eilenberg-Moore \cite{EM}. Assume then that we are given a pseudo-commutative diagram\begin{gather}\label{adjoint}\begin{diagram}[small]\Dd_1&\rTo^\Psi&\Dd_2\\\dTo^{R_1}&{\overset{\phi}{\cong}}&\dTo_{R_2}\\\Ee_1&\rTo_{\Phi}&\Ee_2\end{diagram}\end{gather}with right adjoint functors $R_1,R_2$ and an invertible $2$-cell $\phi:\Phi R_1\cong R_2\Psi$. We denote the left adjoints by $L_1,L_2$ respectively. Such an \emph{adjoint square} will be called \emph{exact} if the adjoint $2$-cell $$\psi=\eps_2\Psi L_1\cdot L_2\phi L_1\cdot L_2\Phi\eta_1:L_2\Phi\Rightarrow\Psi L_1$$is also invertible. It is then straighforward to check that the natural transformation$$\rho=(\phi L_1)^{-1}(L_2\psi):R_2L_2\Phi\Rightarrow\Phi R_1L_1$$ defines a monad morphism $(\Phi,\rho):(\Ee_1,R_1L_1)\to(\Ee_2,R_2L_2)$ and hence a diagram\begin{gather}\label{adjoint2}\begin{diagram}[small]\Dd_1&\rTo^{\Psi}&\Dd_2\\\dTo^{K_1}&\cong&\dTo_{K_2}\\\Ee_1^{R_1L_1}&\rTo^{\overline{\Phi}}&\Ee_2^{R_2L_2}\\\dTo^{U_1}&&\dTo_{U_2}\\\Ee_1&\rTo^{\Phi}&\Ee_2\end{diagram}\end{gather}in which the lower square commutes and the upper square pseudo-commutes. Moreover, gluing of diagram (\ref{adjoint2}) gives back the initial adjoint square (\ref{adjoint}).

\begin{prp}\label{T1}For any exact adjoint square (\ref{adjoint}) with monadic functors $R_1,R_2$ and with fully faithful functor~$\Phi$, the functor $\Psi$ is also fully faithful and the essential image of $\Psi$ is obtained from the essential image of $\Phi$ by pullback along $R_2$.\end{prp}

\begin{proof}Since $\rho$ is invertible and $\Phi$ is fully faithful, Proposition \ref{P1}c shows that in the induced diagram (\ref{adjoint2}) the functor $\overline{\Phi}$ is fully faithful and its essential image is obtained as the pullback along $U_2$ of the essential image of $\Phi$. Since $K_1$ and $K_2$ are fully faithful by monadicity of $R_1$ and $R_2$, and since the upper square pseudo-commutes, $\Psi$ is fully faithful. Next, since isomorphic functors have the same essential image and $K_1$ is essentially surjective (again by monadicity of $R_1$), the essential image of $K_2\Psi$ coincides with the essential image of $\overline{\Phi}$. Finally, since $K_2$ is fully faithful, the essential image of $\Psi$ is obtained from the essential image of $\overline{\Phi}$ by pullback along $K_2$, and thus from the essential image of $\Phi$ by pullback along $U_2K_2=R_2$.\end{proof}

\subsection{Dense generators.}\label{densegenerator}--\vspace{1ex}

An inclusion of a full subcategory $i_\Aa:\Aa\inc\Ee$ is called a \emph{dense generator} of $\Ee$ if $\Aa$ is \emph{small} and the associated \emph{nerve functor} $\nu_\Aa:\Ee\to\PSh\Aa$ is \emph{fully faithful}. Recall that for each object $X$ of $\Ee$, the \emph{$\Aa$-nerve} $\nu_\Aa(X)$ is defined by $\nu_\Aa(X)(A)=\Ee(i_\Aa(A),X)$ where $A$ is an object of $\Aa$; we shall sometimes write $\nu_\Aa=\Ee(i_\Aa,-)$.

For each object $X$ of $\Ee$, we denote by $\Ee/X$ the \emph{slice category} over $X$, whose objects are arrows $Y\to X$ and whose morphisms are commuting triangles over $X$. Let $\Aa/X$ be the full subcategory of $\Ee/X$ spanned by those arrows $A\to X$ which have domain $A$ in $\Aa$.  The canonical projection functor $\Aa/X\to\Ee$ comes equipped with a natural transformation to the constant functor $c_X$. This natural transformation will be called the \emph{$\Aa$-cocone} over $X$. In particular, if $\Ee=\PSh\Aa$, this defines, for any presheaf $X$ on $\Aa$, the classical Yoneda-cocone on $X$, whose diagram is defined on the \emph{category of elements} $\Aa/X$ of $X$. It is well-known that the Yoneda-cocones are colimit-cocones. This property characterises dense generators as asserted by the following equally well-known lemma (cf. \cite[3.5]{GU}):

\begin{lma}A full subcategory $\Aa$ is a dense generator of $\Ee$ if and only if the $\Aa$-cocones in $\Ee$ are colimit-cocones in $\Ee$.\end{lma}

\begin{proof}If the $\Aa$-nerve is fully faithful, it takes the $\Aa$-cocones to the corresponding Yoneda-cocones; since the latter are colimit-cocones in $\PSh\Aa$, the former are colimit-cocones in $\Ee$. Conversely, if $\Aa$-cocones are colimit-cocones, then any map $f$ in $\Ee$ can uniquely be recovered (as a colimit) from its nerve $\nu_\Aa(f)$. Indeed, for any object $X$ of $\Ee$, the category of elements of $\nu_\Aa(X)$ may be identified with $\Aa/X$.\end{proof}

\begin{dfn}\label{monadwitharities}A monad $T$ on a category $\Ee$ with dense generator $\Aa$ is called a monad with arities $\Aa$ if the composite functor $\nu_\Aa T$ takes the $\Aa$-cocones in $\Ee$ to colimit-cocones in $\PSh\Aa$.\end{dfn}

For any monad $T$ on a category $\Ee$ with dense generator $\Aa$, we define the category $\Theta_T$ to be the full subcategory of $\Ee^T$ spanned by the free $T$-algebras on the objects of $\Aa$. The full inclusion $\Theta_T\inc \Ee^T$ will be denoted by $i_T$. There is a uniquely determined functor $j_T:\Aa\to\Theta_T$ such that $i_Tj_T=Fi_\Aa$ where $F:\Ee\to\Ee^T$ is left adjoint to the forgetful functor $U:\Ee^T\to\Ee$. This factorisation of $Fi_\Aa$ into a bijective-on-objects functor $j_T:\Aa\to\Theta_T$ followed by a fully faithful functor $i_T:\Theta_T\to\Ee^T$ can also be used to define the category $\Theta_T$ up to isomorphism. The nerve associated to the full inclusion $i_T$ will be denoted by $\nu_T:\Ee^T\to\PSh{\Theta_T}$.

The following diagram summarises the preceding definitions and notations:
\begin{gather}\label{nerve}\begin{diagram}[small,p=1mm]\Theta_T&\rTo^{i_T}&\Ee^T&\rTo^{\nu_T}&\PSh{\Theta_T}\\\uTo^{j_T}&&\uTo^F\dTo_U&&\dTo_{j_T^*}\\\Aa&\rTo^{i_\Aa}&\Ee&\rTo^{\nu_\Aa}&\PSh\Aa\end{diagram}\end{gather}
For any monad $T$, the left square is commutative by construction, while the right square is an \emph{adjoint square} in the sense of Section \ref{sct:adjoint}. Indeed, $U$ and $j^*_T$ are right adjoint (even \emph{monadic}) functors, and we have the following natural isomorphism: $$\nu_\Aa U=\Ee(i_\Aa,U(-))\cong\Ee^T(Fi_\Aa,-)=\Ee^T(i_Tj_T,-)=j_T^*\nu_T.$$

\begin{prp}\label{P3}A monad~$T$ on a category~$\Ee$ with dense generator~$\Aa$ is a monad with arities~$\Aa$ if and only if the right square in (\ref{nerve}) is an exact adjoint square.\end{prp}
\begin{proof}Notice first that the left adjoint $(j_T)_!$ of $j^*_T$ is given by left Kan extension along $j_T$. In particular, $(j_T)_!$ takes the representable presheaves to representable presheaves so that we have a canonical isomorphism $(j_T)_!\nu_\Aa i_\Aa\cong\nu_Ti_Tj_T=\nu_TFi_\Aa$.

Consider an object $X$ of $\Ee$ equipped with its $\Aa$-diagram $a_X:\Aa/X\to\Ee$ whose colimit is $X$. Since $a_X$ takes values in $\Aa$, the functors $(j_T)_!\nu_\Aa a_X$ and $\nu_TF a_X$ are canonically isomorphic; since the right square is an adjoint square, application of $j_T^*$ induces thus a canonical isomorphism between $j_T^*(j_T)_!\nu_\Aa a_X$ and $\nu_\Aa UF a_X$. 

On the other hand, we know that $\nu_\Aa a_X$ induces the Yoneda-cocone for $\nu_\Aa(X)$ by density of $\Aa$. Therefore, $j_T^*(j_T)_!\nu_\Aa a_X$ induces a colimit-cocone for $j_T^*(j_T)_!(X)$. Since $j_T^*$ is monadic, the right square is an exact adjoint square precisely when $j_T^*(j_T)_!\nu_\Aa(X)$ is canonically isomorphic to $\nu_\Aa UF(X)$ for each object $X$ of $\Ee$. Using the isomorphism above, as well as the isomorphism between $j_T^*(j_T)_!\nu_\Aa a_X$ and $\nu_\Aa UF a_X$, this is the case if and only if $T=UF$ is a monad with arities $\Aa$.\end{proof}

As explained in the introduction, the following theorem has been formulated and proved at different levels of generality by Leinster \cite{L} and the three authors of this article, cf. \cite[1.12/17]{Be}, \cite[4.10]{W}, \cite{M}.

\begin{thm}[Nerve Theorem]\label{T2}Let $\Ee$ be a category with dense generator $\Aa$. For any monad~$T$ with arities $\Aa$, the full subcategory~$\Theta_T$ spanned by the free $T$-algebras on the arities is a dense generator of the Eilenberg-Moore category~$\Ee^T$.

The essential image of the nerve functor $\nu_T:\Ee^T\to\PSh{\Theta_T}$ is spanned by those presheaves whose restriction along $j_T$ belongs to the essential image of $\nu_\Aa:\Ee\to\PSh\Aa$.\end{thm}

\begin{proof}This follows from Propositions \ref{T1} and \ref{P3}.\end{proof}

\begin{rmk}The density of $\Theta_T$ in $\Ee^T$ holds for a larger class of monads than just the monads with arities $\Aa$. Indeed, a careful look at the preceding proof shows that, in virtue of Proposition \ref{P1}b, density follows already if the natural transformation $\rho$, responsible for diagram (\ref{adjoint2}), is a pointwise epimorphism. This property in turn amounts precisely to the hypothesis of Theorem 1.3 of Day \cite{D}, in which the density of $\Theta_T$ in $\Ee^T$ is established by other methods (namely, reducing it to the density of the Kleisli category $\Ee_T$ in the Eilenberg-Moore category $\Ee^T$, cf. \cite[1.2]{D}). 

On the other hand, if we assume the density of~$\Theta_T$, the description (in \ref{T2}) of the essential image of the nerve functor $\nu_T:\Ee^T\to\widehat{\Theta_T}$ is equivalent to $T$ being a monad with arities $\Aa$, as follows from Proposition \ref{P3}. This conditional characterisation of monads with arities $\Aa$ is closely related to Theorem 5.1 of Diers \cite{Di}.\end{rmk}

\subsection{Accessible and locally presentable categories}\label{section:accessible}--\vspace{1ex}

Recall \cite{MP,AR} that a category $\Ee$ is called $\alpha$-\emph{accessible} for a regular cardinal $\alpha$ if $\Ee$ has $\alpha$-filtered colimits and comes equipped with a dense generator $\Aa$ such that\begin{itemize}\item[(i)]the objects of~$\Aa$ are $\alpha$-presentable;\item[(ii)]for each object $X$ of~$\Ee$, the category $\Aa/X$ is $\alpha$-filtered.\end{itemize}In particular, each object of an $\alpha$-accessible category $\Ee$ is a canonical $\alpha$-filtered colimit of $\alpha$-presentable objects. A \emph{cocomplete} $\alpha$-accessible category is called \emph{locally $\alpha$-presentable} \cite{GU}. A cocomplete category $\Ee$ is locally $\alpha$-presentable if and only if $\Ee$ has a \emph{strong generator} of $\alpha$-presentable objects, cf. \cite[7.1]{GU}. This is so since $\alpha$-cocompletion of such a strong generator yields a dense generator satisfying conditions (i) and (ii) above, cf. \cite[7.4]{GU}.

Any $\alpha$-accessible category $\Ee$ has a canonical dense generator, namely $\Aa$ can be chosen to be a skeleton $\Ee(\alpha)$ of the full subcategory spanned by all $\alpha$-presentable objects of $\Ee$. In particular, the latter is essentially small. Since $\Ee(\alpha)$ consists of $\alpha$-presentable objects, the nerve functor $\nu_{\Ee(\alpha)}:\Ee\to\PSh{\Ee(\alpha)}$ preserves $\alpha$-filtered colimits. In particular, any monad $T$ which preserves $\alpha$-filtered colimits is a monad with arities $\Ee(\alpha)$. Moreover, the essential image of $\nu_{\Ee(\alpha)}$ is spanned by the \emph{$\alpha$-flat presheaves} on $\Ee(\alpha)$, i.e. those presheaves whose category of elements is $\alpha$-filtered.

The following theorem has been proved by Gabriel-Ulmer \cite[10.3]{GU} under the additional hypothesis that $\Ee$ is cocomplete. Ad\'amek-Rosick\'y \cite[2.78]{AR} obtain the first half of the theorem by quite different methods. Our proof shows that the degree of accessibility is preserved under passage to the Eilenberg-Moore category.

\begin{thm}[Gabriel-Ulmer, Ad\'amek-Rosick\'y]\label{T3}For any $\alpha$-filtered colimit preserving monad~$T$ on an $\alpha$-accessible category~$\Ee$, the Eilenberg-Moore category~$\Ee^T$ is $\alpha$-accessible. The dense generator~$\Theta_T$ of $\Ee^T$ is spanned by the free $T$-algebras on (a skeleton of) the $\alpha$-presentable objects. Moreover, $\Ee^T$ is equivalent to the full subcategory of $\PSh{\Theta_T}$ spanned by those presheaves whose restriction along $j_T$ is $\alpha$-flat.\end{thm}

\begin{proof}Since $T$ is a monad with arities $\Ee(\alpha)$, Theorem \ref{T2} yields the density of $\Theta_T$ and the description of $\Ee^T$ as full subcategory of $\PSh{\Theta_T}$. Moreover, $\Ee^T$ has $\alpha$-filtered colimits, since $\Ee$ has and $T$ preserves them. The forgetful functor $U:\Ee^T\to\Ee$ also preserves $\alpha$-filtered colimits; therefore the left adjoint $F:\Ee\to\Ee^T$ takes $\alpha$-presentable objects to $\alpha$-presentable objects; this establishes condition (i) of $\alpha$-accessibility; condition (ii) follows from an adjunction argument.\end{proof}

\section{Monads with arities.}

\subsection{Categories with arities and arity-respecting functors}\label{section:arity-respecting}Monads with arities can be described as monads in a certain $2$-category (cf. \cite{St0}) which deserves some interest for itself. The objects of the $2$-category relevant to us are \emph{categories with arities}, i.e. pairs $(\Ee,\Aa)$ consisting of a category $\Ee$ with dense generator $\Aa$, cf. Section \ref{densegenerator}. The $1$-cells $(\Ee,\Aa)\to(\Ff,\Bb)$ are \emph{arity-respecting functors}, the $2$-cells are natural transformations.

Here, a functor $F:(\Ee,\Aa)\to(\Ff,\Bb)$ is called \emph{arity-respecting} if the composite functor $\nu_BF$ takes the $\Aa$-cocones in $\Ee$ to colimit-cocones in $\PSh\Bb$, cf. Definition \ref{monadwitharities}. Observe that $F$ is arity-respecting if and only if $\nu_\Bb F$ is canonically isomorphic to the left Kan extension along $i_A$ of its restriction $\nu_BFi_A$. The following lemma shows that we get indeed a $2$-category in this way. It is worthwhile to note that the natural transformations between two parallel arity-respecting functors form a \emph{set} since our dense generators are \emph{small} by definition.

\begin{lma}\label{L3}--
\begin{enumerate}\item[(a)]Identity functors are arity-respecting;\item[(b)] The composition of arity-respecting functors is arity-respecting;\item[(c)] A functor $F:(\Ee,\Aa)\to(\Ff,\Bb)$ respects arities if and only if for any left Kan extension $\overline{F}:\PSh\Aa\to\PSh\Bb$ of~$\nu_{\Bb}Fi_{\Aa}$ along the Yoneda embedding $y_{\Aa}:\Aa\to\PSh\Aa$ there is an invertible $2$-cell $\phi:\nu_{\Bb}F\cong\overline{F}\nu_{\Aa}$.\end{enumerate}\end{lma}
\begin{proof}--

(a) This is just a reformulation of the density of the generator;

(b) Let $F:(\Ee_1,\Aa_1)\to(\Ee_2,\Aa_2)$ and $G:(\Ee_2,\Aa_2)\to(\Ee_3,\Aa_3)$ be arity-respecting functors and define $\overline{F}:\PSh\Aa_1\to\PSh\Aa_2$ and $\overline{G}:\PSh\Aa_2\to\PSh\Aa_3$ as left Kan extensions (along the Yoneda embedding) of $\nu_{\Aa_2}Fi_{\Aa_1}$ and $\nu_{\Aa_3}Gi_{\Aa_2}$ respectively. Assuming (c), this yields the following pseudo-commutative diagram of functors\begin{diagram}[small]\Ee_1&\rTo^{\nu_{\Aa_1}}&\PSh\Aa_1\\\dTo^F&\overset{\phi}{\cong}&\dTo_{\overline{F}}\\\Ee_2&\rTo^{\nu_{\Aa_2}}&\PSh\Aa_2\\\dTo^G&\overset{\psi}{\cong}&\dTo_{\overline{G}}\\\Ee_3&\rTo^{\nu_{\Aa_3}}&\PSh\Aa_3\end{diagram}in which the $2$-cells $\phi$ and $\psi$ are invertible. The functor $\overline{G}$ has a right adjoint by construction so that post-composition with $\overline{G}$ preserves left Kan extensions; in particular, $\overline{G}\circ\overline{F}$ is a left Kan extension of $\nu_{\Aa_3}GFi_{\Aa_1}$ along $y_{\Aa_1}$. Since gluing of $\phi$ and $\psi$ along $\nu_{\Aa_2}$ yields an invertible $2$-cell between $\overline{G}\overline{F}\nu_{\Aa_1}$ and $\nu_{\Aa_3}GF$, another application of (c) implies that $GF$ respects arities as required.

(c) By construction, $\overline{F}$ has a right adjoint and thus preserves left Kan extensions. Therefore (by density of $\Aa$) the left Kan extension of $\overline{F}y_\Aa$ along $i_\Aa$ is $2$-isomorphic to $\overline{F}\nu_\Aa$. Assume first that the left Kan extension of $\nu_\Bb Fi_\Aa$ along $i_\Aa$ is $2$-isomorphic to $\nu_\Bb F$. Since by definition $\nu_\Bb Fi_\Aa=\overline{F}\nu_\Aa i_\Aa=\overline{F}y_\Aa$, this implies that the left Kan extension of $\overline{F}y_\Aa$ along $i_\Aa$ is $2$-isomorphic to $\nu_\Bb F$, whence the required $2$-isomorphism $\phi:\nu_\Bb F\cong\overline{F}\nu_\Aa$. Conversely, if such an invertible $2$-cell exists, the left Kan extension of $\nu_\Bb Fi_\Aa$ along $i_\Aa$ is canonically $2$-isomorphic to $\nu_\Bb F$.\end{proof}

\begin{prp}\label{prp:EM-object}A monad~$T$ on a category~$\Ee$ with dense generator~$\Aa$ has arities $\Aa$ if and only if~$T$ is an arity-respecting endofunctor of~$(\Ee,\Aa)$. If this is the case, the pair $(\Ee^T,\Theta_T)$ is an Eilenberg-Moore object of~$T$ in the $2$-category of categories with arities, arity-respecting functors and natural transformations.\end{prp}

\begin{proof}By definition, the monad $T$ has arities $\Aa$ if and only if $T$ is arity-respecting. Theorem \ref{T2} implies that $(\Ee^T,\Theta_T)$ is a category with arities. Since~$\Ee^T$ is an Eilenberg-Moore object of~$T$ in the ordinary $2$-category of categories, functors and natural transformations, it remains to be shown that the free and forgetful functors respect arities. This follows from \ref{P3} and \ref{L3}c using that $(j_T)_!$ (resp. $j_T^*$) is the left Kan extension of $\nu_TFi_\Aa$ (resp. $\nu_AUi_T$) along $y_\Aa$ (resp. $y_T$).\end{proof}

\subsection{Factorisation categories}
Let $(\Ee,\Aa)$ and $(\Ff,\Bb)$ be categories with arities and let $F:\Ee\to\Ff$ be a functor. For an elementary formulation of what it means for $F$ to respect arities, we introduce the following factorisation categories. For any morphism $f:B \to FX$ (where $B$ is an object of $\Bb$ and $X$ is an object of $\Ee$) the factorisation category $\Fact{\Aa}{F}{f}$ is defined as follows:

An object is a triple $(g,A,h)$ as in $\xygraph{{B}="l" [r] {FA}="m" [r] {FX}="r" "l":"m"^-{g}:"r"^-{Fh}}$ such that the composite is $f$ and $A$ is an object of $\Aa$. A morphism $(g_1,A_1,h_1) \to (g_2,A_2,h_2)$ consists of a morphism $k:A_1 \to A_2$ in $\Aa$ such that $F(k)g_1=g_2$ and $h_2k=h_1$. Later on we shall also use the notation $\Fact{\Ee}{F}{f}$ if no restriction is made on the object $A$.

\begin{prp}\label{prp:monad-with-arities-elementary}
The functor $F:(\Ee,\Aa)\to(\Ff,\Bb)$ respects arities if and only if, for all $f:B \to FX$ as above, the factorisation category $\Fact{\Aa}{F}{f}$ is connected.\end{prp}

\begin{proof}
By definition, $F$ respects arities if and only if $\nu_{\Bb}F$ preserves the $\Aa$-cocones for all objects $X$ of $\Ee$. Since the evaluation functors $\tn{ev}_B:\PSh\Bb\to\Sets$ (for $B$ running through the objects of $\Bb$) collectively preserve and reflect colimits, this is in turn equivalent to saying that the functions
\begin{equation}\label{eq:colimit-cocone-monad-with-arity-charn}
\ca E(B,Fh) : \ca E(B,FA) \to \ca E(B,FX)
\end{equation}
varying over $h:A \to X$ in $\Aa/X$ form a colimit-cocone. Because a colimit of a set-valued functor is computed as the set of connected components of the category of its elements, it follows that the fibre of the induced function
\begin{equation}\label{eq:induced-map-monad-with-arity-charn}
\colim\limits_{h:A \to X} \ca E(B,FA) \to \ca E(B,FX)
\end{equation}
over $f:B \to FX$ is given by the connected components of the category $\Fact{\ca A}{F}{f}$. To say that (\ref{eq:colimit-cocone-monad-with-arity-charn}) is a colimit is to say that for all $f$ (\ref{eq:induced-map-monad-with-arity-charn}) is a bijection, which is equivalent to saying that these fibres are singletons.\end{proof}

\subsection{Local right adjoints and generic factorisations}\label{dfn:generic}

A functor $R:\Ee\to\Ff$ is said to be a \emph{local right adjoint} if for each object $X$ of~$\Ee$ the induced functor$$R_X : \Ee/X \to \Ff/RX\quad\quad f \mapsto R(f)$$admits a left adjoint functor $L_X:\Ff/RX\to\Ee/X$. If $\Ee$ has a terminal object $1$ it suffices to require that $R_1$ has a left adjoint $L_1$, cf. Lemma \ref{lma:lra} below.

In \cite{W0,W}, the terminology \emph{parametric} right adjoint was used instead, following Street \cite{St}, but since then local right adjoint has become the more accepted terminology. A functor between presheaf categories is local right adjoint if and only if it \emph{preserves connected limits} if and only if it is \emph{familially representable}, cf. \cite[C.3.2]{L0}.


A morphism $g:B \to RA$ is said to be \emph{$R$-generic} whenever, given $\alpha$, $\beta$ and $\gamma$ making
\[ \xygraph{!{0;(1.5,0):(0,.667)::} {B}="tl" [r] {RA'}="tr" [d] {RX}="br" [l] {RA}="bl" "tl":"tr"^-{\alpha}:"br"^-{R\gamma}:@{<-}"bl"^-{R\beta}:@{<-}"tl"^-{g}} \]
commute, there is a unique $\delta:A \to A'$ such that $R(\delta)g=\alpha$ and $\beta=\gamma\delta$.

The following lemma is a reformulation of \cite[2.6]{W} and \cite[5.9]{W0}.

\begin{lma}\label{lma:lra}Assume that $\Ee$ has a terminal object $1$. For a functor $R:\Ee\to\Ff$ the following conditions are equivalent:\begin{itemize}\item[(i)]$R$ is a local right adjoint;\item[(ii)]for each $f:B\to RX$, the category $\Fact{\Ee}{R}{f}$ has an initial object;\item[(iii)]for each $\tilde{f}:B\to R1$, the category $\Fact{\Ee}{R}{\tilde{f}}$ has an initial object;\item[(iv)]the functor $R_1:\Ee\to\Ff/R1$ has a left adjoint $L_1:\Ff/R1\to\Ee$;\item[(v)]each $f:B\to RX$ factors as $\xygraph{{B}="l" [r] {RA}="m" [r] {RX}="r" "l":"m"^-{g}:"r"^-{Rh}}$ where $g$ is $R$-generic.\end{itemize}\end{lma}

\begin{proof}$R_X$ admits a left adjoint $L_X$ if and only if, for each $f:Y\to RX$ (considered as an object of $\Ff/RX$), the comma category $R_X/f$ has an initial object. This comma category $R_X/f$ may be identified with the factorisation category $\Fact{\Ee}{R}{f}$. It follows that condiditons (i) and (ii) are equivalent, and conditions (iii) and (iv) are equivalent.  Clearly condition (ii) implies condition (iii). A factorisation of $f$ as in (v) is an initial object of $\Fact{\Ee}{R}{f}$ so that (v) implies (ii). It thus remains to be shown that (iii/iv) implies (v).

For this, denote by $t_X:X\to 1$ the unique existing map, and factor $R(t_X)f$ through the initial object of $\Fact{\Ee}{R}{R(t_X)f}$ to obtain the following diagram

\[\xygraph{!{0;(1.5,0):(0,.667)::} {B}="tl" [r] {RX}="tr" [d] {R1}="br" [l] {RA}="bl" "tl":"tr"^-{f}:"br"^-{R(t_X)}:@{<-}"bl"^-{R(t_A)}:@{<-}"tl"^-{g} "bl":@{.>}"tr"|-{Rh}}\]

Observe that $h:A\to X$ may be identified with $L_1(f)$ and $g:B\to RA$ with the unit $\eta_f$ of the $(L_1,R_1)$-adjunction at $f$. Since any factorisation of $f$ as  $\xygraph{{B}="l" [r] {RA'}="m" [r] {RX}="r" "l":"m"^-{g'}:"r"^-{Rh'}}$ can be considered in an obvious way as a factorisation of $R(t_X)f$ in $\Ff/R1$, the universal property of $\eta_f$ yields the map from $(g,A,h)$ to $(g',A',h')$ in $\Fact{\Ee}{R}{f}$ required for the $R$-genericity of $g$.\end{proof}

\subsection{Cartesian and strongly cartesian monads}\label{dfn:cartesian}

Recall that a monad $T$ on a category $\Ee$ with pullbacks is called \emph{cartesian} if $T$ preserves pullbacks, and if all naturality squares of unit and multiplication of~$T$ are pullbacks. A cartesian monad is called \emph{strongly cartesian} if the underlying endofunctor is a local right adjoint.

For a given cartesian monad $T$, any monad $S$ on~$\Ee$ equipped with a cartesian monad morphism $S\Rightarrow T$ will be called \emph{$T$-cartesian}. $T$-cartesian monads are themselves cartesian monads. If $T$ is strongly cartesian, $T$-cartesian monads are strongly cartesian as well (by \ref{lma:lra} and the fact that for cartesian monad morphisms $S\Rightarrow T$, $T$-generic factorisations induce $S$-generic factorisations by pullback).

Any $T$-cartesian monad $S$ determines, and is up to isomorphism uniquely determined by, a \emph{$T$-operad} in the sense of Leinster \cite[4.2.3]{L0}. Indeed, the \emph{$T$-collection} underlying $S$ is simply the morphism $S1\to T1$ induced by evaluation at a terminal object $1$ of $\Ee$. The monad structure of $S$ over $T$ amounts then to the $T$-operad structure of the $T$-collection $S1$ over $T1$ because of the cartesianness of $S\Rightarrow T$.

A generator $\Aa$ of $\Ee$ is called \emph{$T$-generically closed} if for any $T$-generic morphism $B\to TA$ with $B$ in $\Aa$, there is an object isomorphic to $A$ which belongs to $\Aa$.

\begin{thm}\label{thm:lra->arities}Let~$T$ be a strongly cartesian monad on a finitely complete category~$\Ee$. Any dense generator $\Aa$ of~$\Ee$ embeds in a minimal $T$-generically closed dense generator $\Aa_T$. In particular, any~$T$-cartesian monad on~$\Ee$ has arities $\Aa_T$.\end{thm}

\begin{proof}Up to isomorphism, the objects of~$\Aa_T$ may be obtained in the following way. Denote by $1$ a terminal object of $\Ee$. Take any $f:B\to T1$ with domain $B$ in $\Aa$, and then, according to \ref{lma:lra}(v), generically factor $f$ to obtain a $T$-generic morphism $g:B\to TA$. Define $\Aa_T$ to be the full subcategory of $\Ee$ spanned by all objects $A$ obtained in this way. Since $\Aa$ is small (and $\Ee$ locally small), $\Aa_T$ is small as well. By \cite[5.10.2]{W0}, the components of the unit $\eta$ of $T$ are generic, and so $\Aa_T$ contains $\Aa$; in particular, $\Aa_T$ is a dense generator of $\Ee$, cf. \cite[3.9]{GU}.

For any $T$-generic morphism $g:B \to TA$ with domain $B$ in $\Aa_T$, there is a $T$-generic morphism $g':C \to TB$ with domain $C$ in $\Aa$. In particular, the composite
\[ \xygraph{{C}="p1" [r] {TB}="p2" [r] {T^2A}="p3" [r] {TA}="p4" "p1":"p2"^-{g'}:"p3"^-{Tg}:"p4"^-{\mu_A}} \]
is $T$-generic by \cite[5.14 and 5.10.2]{W0}, so that $A$ belongs to $\Aa_T$, i.e. $\Aa_T$ is $T$-generically closed. It follows then from \cite[5.10.2]{W0} that $\Aa_T$ is $S$-generically closed for any $T$-cartesian monad $S$ on $\Ee$. Hence, given any $f:B \to SX$ with domain $B$ in $\Aa_T$, an $S$-generic factorisation \ref{lma:lra}(v) for it
\[ \xygraph{{B}="l" [r] {SA}="m" [r] {SX}="r" "l":"m"^-{g}:"r"^-{Sh}} \]
may be regarded as an initial object in $\Fact{\Aa_T}{S}{f}$. It follows then from Propositions \ref{prp:EM-object} and \ref{prp:monad-with-arities-elementary} that $S$ has arities $\Aa_T$.\end{proof}

\begin{rmk}\label{rmk:canonical arities}The most important situation in which Theorem \ref{thm:lra->arities} has been applied so far is when $\Ee$ is a presheaf category $\PSh{\CC}$ and $\Aa=\CC$ consists of the representable presheaves. For instance in \cite[4.16]{W} the objects of $\CC_T$ in just this situation were called $T$-cardinals. In the present article, we shall call $\CC_T$ the \emph{canonical arities} for the strongly cartesian monad $T$. Alternatively, these canonical arities are those presheaves which belong to the essential image of the composite functor
\[ \xygraph{!{0;(2.5,0):(0,1)::}{y_\CC/T1}="l" [r] {\PSh{y_{\CC}/T1}\catequiv \PSh{\CC}/T1 } ="m" [r] {\PSh{\CC}}="r" "l":"m"^-{\tn{yoneda}}:"r"^-{L_1}} \]
where $L_1$ is left adjoint to $T_1$, cf. the proof of Lemma \ref{lma:lra} and \cite[2.11]{W}.\end{rmk}

\subsection{Globular operads}\label{dfn:omega-operads}There are many interesting examples of cartesian monads, cf. \cite[section 2]{W} or the work of Kock \cite{K} and Joyal-Kock \cite{JK}. We discuss here the \emph{$\omega$-operads} of Batanin \cite[7.1]{Ba} since they motivated many ideas of this article.

Starting point is Batanin's observation \cite[4.1.1]{Ba} that the free $\omega$-category monad $D_\omega$ on the category of globular sets is cartesian. Street \cite{St} and Leinster \cite{L0} observe that Batanin's concept of an $\omega$-operad amounts to the concept of a $D_\omega$-operad, cf. Section \ref{dfn:cartesian}, i.e. each $\omega$-operad induces a $D_\omega$-cartesian monad on globular sets. Similarily, Batanin's \emph{$n$-operads} are $D_n$-operads and induce thus  $D_n$-cartesian monads on $n$-globular sets, where $D_n$ denotes the free $n$-category monad.

It turns out that $D_\omega$ is a strongly cartesian monad so that the monads induced by $\omega$-operads have canonical arities by Theorem \ref{thm:lra->arities} and Remark \ref{rmk:canonical arities}. These canonical arities have been constructed by Batanin \cite[pg. 62]{Ba}, and called \emph{globular cardinals} by Street \cite[pg. 311]{St}, resp. \emph{globular pasting diagrams} by Leinster \cite[8.1]{L0}. We shall use the notation $\Theta_0$ for these arities, following \cite[1.5]{Be} which contains a description of $\Theta_0$ as full subcategory of globular sets. There are truncated versions $\Theta_{n,0}$ of $\Theta_0$ which serve as canonical arities for $D_n$. For instance, the canonical arities $\Theta_{1,0}$ for the free category monad $D_1$ consist of those graphs which represent directed edge-paths of finite length. As a category, $\Theta_{1,0}$ may be identified with the subcategory $\Delta_0$ of the simplex category $\Delta$ having same objects as $\Delta$ but only those simplicial operators $\phi:[m]\to[n]$ which satisfy $\phi(i+1)=\phi(i)+1$ for $0\leq i<m$.

Theorem \ref{T2} applied to the free category monad $D_1$ yields the characterisation of small categories as simplicial sets satisfying the Grothendieck-Segal conditions \cite{G,Se0}, see \cite[1.13]{Be}, \cite[2.8]{W}, \cite{M} for details. There are analogous characterisations of $\omega$-categories (or more generally: algebras over $\omega$-operads) as nerves subject to generalised Grothendieck-Segal conditions, cf. \cite[1.12/17]{Be}, \cite[4.26]{W}.

\section{Theories with arities.}

\begin{dfn}[cf. \cite{M}]\label{dfn:theory}Let~$\Ee$ be a category with dense generator~$\Aa$.

A \emph{theory} $(\Theta,j)$ with arities $\Aa$ on $\Ee$ is a bijective-on-objects functor $j:\Aa\to\Theta$ such that the induced monad $j^*j_!$ on $\widehat{\Aa}$ preserves the essential image of~$\nu_A:\Ee\to\widehat{\Aa}$.

A \emph{$\Theta$-model} is a presheaf on $\Theta$ whose restriction along $j$ belongs to the essential image of $\nu_\Aa$.\end{dfn}

A \emph{morphism of $\Theta$-models} is just a natural transformation of the underlying presheaves. The category of $\Theta$-models will be denoted $\Mod_\Theta$, and is thus a full subcategory of the presheaf category $\widehat{\Theta}$.

A \emph{morphism of theories} $(\Theta_1,j_1)\to(\Theta_2,j_2)$ is a functor $\theta:\Theta_1\to\Theta_2$ such that $j_2=\theta j_1$. We shall write $\Th(\Ee,\Aa)$ for the category of theories with arities $\Aa$ on $\Ee$, and $\Mnd(\Ee,\Aa)$ for the category of monads with arities $\Aa$ on $\Ee$. Observe that monad morphisms $\rho:T_1\Rightarrow T_2$ point here in the opposite direction than in Section \ref{dfn:monad-morphism} where we adopted the convention $(Id_\Ee,\rho):(\Ee,T_2)\to(\Ee,T_1)$.

\begin{prp}\label{prp:theory-models}Let~$T$ be a monad with arities~$\Aa$ on a category~$\Ee$. Let~$\Theta_T$ be the full subcategory of $\Ee^T$ spanned by the free $T$-algebras on the objects of~$\Aa$, and let $j_T$ be the (restricted) free $T$-algebra functor.

Then, $(\Theta_T,j_T)$ is a theory with arities $\Aa$ on $\Ee$. The algebraic nerve functor induces an equivalence between the categories of~$T$-algebras and of~$\Theta_T$-models.\end{prp}

\begin{proof}Since, by Proposition \ref{P3}, the right square of diagram (\ref{nerve}) is an exact adjoint square, the monad $j_T^*(j_T)_!$ on $\widehat{\Aa}$ preserves the essential image of $\nu_A$. Therefore, $(\Theta_T,j_T)$ is a theory with arities $\Aa$. Theorem \ref{T2} implies then that $\Mod_{\Theta_T}$ is the essential image of the fully faithful algebraic nerve functor $\nu_T:\Ee^T\to\PSh{\Theta_T}$.\end{proof}

\begin{lma}\label{lma:theory-morphism}There is a canonical one-to-one correspondence between theory morphisms $(\Theta_1,j_1)\to(\Theta_2,j_2)$ and monad morphisms $j_1^*(j_1)_!\Rightarrow j_2^*(j_2)_!$.\end{lma}

\begin{proof}Let $\theta:\Theta_1\to\Theta_2$ be a functor such that $j_2=\theta j_1$. In particular, the monad $j_2^*(j_2)_!$ may be identified with $j_1^*\theta^*\theta_!(j_1)_!$, whence the unit of the $(\theta_!,\theta^*)$-adjunction induces a monad morphism $j_1^*(j_1)_!\Rightarrow j_2^*(j_2)_!$. Conversely, a given monad morphism $j_1^*(j_1)_!\Rightarrow j_2^*(j_2)_!$ induces a functor of Kleisli categories $\PSh\Aa_{j_1^*(j_1)_!}\to\PSh\Aa_{j_2^*(j_2)_!}$. Moreover, for any theory $(\Theta,j)$, the forgetful functor $j^*$ is monadic, i.e. the presheaf category $\PSh\Theta$ is equivalent to the Eilenberg-Moore category $\PSh\Aa^{j^*j_!}$, and hence the category $\Theta$ is isomorphic to the Kleisli category $\PSh\Aa_{j^*j_!}$. Therefore, any monad morphism $j_1^*(j_1)_!\Rightarrow j_2^*(j_2)_!$ induces a functor $\theta:\Theta_1\to\Theta_2$ such that $j_2=\theta j_1$.

The two constructions are mutually inverse.\end{proof}

\begin{thm}[cf. \cite{M}]\label{thm:monad-theory}Let~$\Ee$ be a category with dense generator~$\Aa$. The assignment $T\mapsto(\Theta_T,j_T)$ induces an adjoint equivalence between the category of monads with arities $\Aa$ and the category of theories with arities $\Aa$.\end{thm}
\begin{proof}We first show that the assignment $T\mapsto(\Theta_T,j_T)$ extends to a functor $\Theta:\Mnd(\Ee,\Aa)\to\Th(\Ee,\Aa)$. By definition, the theory $(\Theta_T,j_T)$ embeds in the Eilenberg-Moore category $\Ee^T$ \emph{via} the Kleisli category $\Ee_T$. Any monad morphism $\phi:S\Rightarrow T$ induces a functor of Kleisli categories $\Ee_S\to\Ee_T$; the latter restricts to the required morphism of theories $\Theta_\phi:(\Theta_S,j_S)\to(\Theta_T,j_T)$. This definition is clearly functorial in monad morphisms.

We next show that $\Theta$ admits a right adjoint $M$. By definition of a theory $(\Theta,j)$, the monad $j^*j_!$ on $\widehat{\Aa}$ restricts to the essential image $\EssIm(\nu_\Aa)$. The choice of a right adjoint $\rho_\Aa:\EssIm(\nu_\Aa)\to\Ee$ to the equivalence $\nu_\Aa:\Ee\to\EssIm(\nu_\Aa)$ induces a monad $\rho_\Aa j^*j_!\nu_\Aa$ on $\Ee$; this monad has arities $\Aa$ on $\Ee$, since the monad $j^*j_!$ has arities $\Aa$ on $\widehat{\Aa}$. The assignment $(\Theta,j)\mapsto \rho_\Aa j^*j_!\nu_\Aa$ extends in a canonical way to a functor $M:\Th(\Ee,\Aa)\to\Mnd(\Ee,\Aa)$. We have to show that for any monad $T$ and theory $(\Theta,j)$ with arities $\Aa$, monad morphisms $T\Rightarrow \rho_\Aa j^*j_!\nu_\Aa$ are in binatural one-to-one correspondence with theory morphisms $(\Theta_T,j_T)\to(\Theta,j)$, or equivalently (according to Lemma \ref{lma:theory-morphism}), with monad morphisms $j_T^*(j_T)_!\Rightarrow j^*j_!$.

By adjunction, monad morphisms $T\Rightarrow \rho_\Aa j^*j_!\nu_\Aa$ correspond  bijectively to $2$-cells $\nu_A T\Rightarrow j^*j_!\nu_\Aa$ satisfying the identities of Section \ref{dfn:monad-morphism}. Since $T$ has arities $\Aa$, Proposition \ref{P3} implies the existence of an invertible $2$-cell $j_T^*(j_T)_!\nu_\Aa\cong\nu_A T$ so that we get a bijective correspondence between monad morphisms  $T\Rightarrow \rho_\Aa j^*j_!\nu_\Aa$ and those $2$-cells $j_T^*(j_T)_!\nu_\Aa\Rightarrow j^*j_!\nu_A$ which are compatible with the monad structures of $j_T^*(j_T)_!$ and $j^*j_!$. Since these monads on $\PSh\Aa$ preserve colimits, they coincide (up to canonical isomorphism) with the left Kan extension (along $y_\Aa$) of their restriction to $\Aa$. Moreover, as well $j_T^*(j_T)_!\nu_\Aa$ as well $j^*j_!\nu_\Aa$ are arity-respecting functors $(\Ee,\Aa)\to(\PSh\Aa,\Aa)$. It follows then from Lemma \ref{L3}c and the evident identity $y_\Aa=\nu_\Aa i_\Aa$ that $2$-cells $j_T^*(j_T)_!\nu_\Aa\Rightarrow j^*j_!\nu_A$ correspond bijectively to $2$-cells $j_T^*(j_T)_!\Rightarrow j^*j_!$ as required.

We finally show that the $(\Theta,M)$-adjunction is an adjoint equivalence. For this, observe that the unit of the $(\Theta,M)$-adjunction is invertible by Proposition \ref{P3}. On the other hand, the right adjoint $M$ is full and faithtful by Lemma \ref{lma:theory-morphism}, i.e. the counit of $(\Theta,M)$-adjunction is invertible as well.\end{proof}

\subsection{Algebraic theories}\label{sct:algebraictheory}Lawvere's algebraic theories \cite{La} can be considered as theories in the sense of Definition \ref{dfn:theory} for $\Ee$ the category of sets and $\Aa$ a skeleton of the full subcategory of finite sets. Indeed, we are in the situation of Section \ref{section:accessible} with $\alpha$ being the countable cardinal so that a monad $T$ has arities $\Aa$ if and only if $T$ is finitary (i.e. preserves filtered colimits). On the other hand, a theory with arities $\Aa$ is by definition a bijective-on-objects functor $j:\Aa\to\Theta$ such that $j^*j_!$ preserves the essential image of $\nu_\Aa:\Ee\to\PSh\Aa$, i.e. flat presheaves on $\Aa$. The latter condition can be expressed in more familiar terms: since flat presheaves are filtered colimits of representable presheaves, and since $j^*j_!$ preserves colimits, it suffices to require that $j^*j_!$ takes representable presheaves to flat presheaves. This in turn means that the representable presheaves on $\Theta$ should be flat when restricted to $\Aa$, i.e. they should take coproducts in $\Aa$ to products in sets. Therefore, a theory with arities $\Aa$ is a bijective-on-objects functor $j:\Aa\to\Theta$ which preserves the coproduct-structure of $\Aa$. This is precisely (the dual of) an algebraic theory in the sense of Lawvere; moreover, $\Theta$-models in the sense of Definition \ref{dfn:theory} coincide with models of the algebraic theory $\Theta^\op$ in Lawvere's sense. Theorem \ref{thm:monad-theory} recovers thus the classical correspondence between finitary monads on sets and  Lawvere's algebraic theories. This correspondence yields quite directly that categories of algebras (over sets) for a finitary monad can be characterised (cf. \cite{La}, \cite[chapter 11]{GU}) as being Barr-exact categories admitting a finitely presentable, regular, projective generator (together with its coproducts).\vspace{2ex}

The preceding discussion of algebraic theories reveals how important it is to get hold of the essential image of the nerve functor $\nu_\Aa:\Ee\to\PSh\Aa$. We shall single out a particular case in which this essential image can be described combinatorially, namely the case where~$\Ee$ is a presheaf category $\PSh\CC$ such that $\Aa$ contains the category $\CC$ of representable presheaves. In this case, the nerve functor $\nu_\Aa:\PSh\CC\to\PSh\Aa$ may be identified with the right Kan extension along the inclusion $\CC\inc\Aa$. Therefore, the essential image of the nerve may be obtained by factoring $\nu_\Aa$ into a surjection $\PSh\CC\to\Sh(\Aa,J)$ followed by an embedding $\Sh(\Aa,J)\to\PSh\Aa$, for a uniquely determined Grothendieck topology $J$ on $\Aa$, cf. \cite[VII]{MM}. Thus (cf. \cite[4.14]{W}),

\begin{lma}\label{lma:essential-image}Let~$\Aa$ be a dense generator of~$\PSh\CC$ containing~$\CC$. For a presheaf~$X$ on~$\Aa$, the following three conditions are equivalent:\begin{itemize}\item[(i)]$X$ belongs to the essential image of the nerve functor~$\nu_\Aa:\PSh\CC\to\PSh\Aa$;\item[(ii)]$X$ is a sheaf for the image-topology $J$ of the geometric morphism $\nu_\Aa$;\item[(iii)]$X$ takes the $\CC$-cocones of the objects of $\Aa$ in $\PSh\CC$ to limit-cones in sets.\end{itemize}In particular, a bijective-on-objects functor $j:\Aa\to\Theta$ is a theory with arities $\Aa$ on $\PSh\CC$ if and only if the monad $j^*j_!$ on $\PSh\Aa$ preserves $J$-sheaves.\end{lma}

\begin{proof}We have seen that (i) and (ii) are equivalent. The equivalence of (ii) and (iii) follows from the fact that the $\CC$-cocones of the objects of $\Aa$ are minimal covering sieves for the Grothendieck topology $J$ on $\Aa$, and thus generate $J$.\end{proof}

\begin{rmk}\label{rmk:Paul-Andre}Although more transparent, condition \ref{lma:essential-image}(ii) is still difficult to handle in practice. In particular, in order to check that a bijective-on-objects functor $j:\Aa\to\Theta$ is a theory, it is in general \emph{insufficient} to verify just that $j^*j_!$ takes the representable presheaves to $J$-sheaves (or, what amounts to the same, that the representable presheaves on $\Theta$ are $\Theta$-models). This is due to the fact that $J$-sheaves \emph{cannot} in general be characterised as a certain kind of colimits of representable presheaves, like in the case of algebraic theories, see \cite[Appendix III]{M} for an instructive example. The following \emph{relative} criterion is therefore useful:\end{rmk}

\begin{lma}\label{lma:useful}Assume that $(\Theta_2,j_2)$ is a theory with arities $\Aa$ on $\PSh\CC$, and that $j_1:\Aa\to\Theta_1$ is a bijective-on-objects functor, equipped with a \emph{cartesian} monad morphism $j_1^*(j_1)_!\Rightarrow j_2^*(j_2)_!$. Then $(\Theta_1,j_1)$ is a theory with arities $\Aa$ on $\PSh\CC$ if and only if the monad $j_1^*(j_1)_!$ takes the terminal presheaf on $\Aa$ to a $J_1$-sheaf.\end{lma}

\begin{proof}The necessity of the condition is clear. For its sufficiency, observe that the hypothesis on $(\Theta_1,j_1)$ implies that for all presheaves $X$ on $\Aa$, the following square\begin{diagram}[small,silent]j_1^*(j_1)_!(X)&\rTo&a_{J_1}j_1^*(j_1)_!(X)\\\dTo&&\dTo\\j_2^*(j_2)_!(X)&\rTo&a_{J_2}j_2^*(j_2)_!(X)\end{diagram}is cartesian, where $a_{J_1}$ (resp. $a_{J_2}$) denotes $J_1$- (resp. $J_2$-) sheafification. Therefore, since the monad $j_2^*(j_2)_!$ preserves $J_2$-sheaves by Lemma \ref{lma:essential-image}, the monad $j_1^*(j_1)_!$ preserves $J_1$-sheaves, whence $(\Theta_1,j_1)$ is a theory with arities $\Aa$.\end{proof}

We introduce the following terminology for any theory $(\Theta,j)$ on $(\Ee,\Aa)$:  the morphisms in the image of $j$ are called \emph{free}; a morphism $g$ in $\Theta$ is called \emph{generic} if for each factorisation $g=j(f)g'$, $f$ is invertible. In other words, a morphism is generic if it factors through free morphisms only if they are invertible in $\Aa$.

\begin{dfn}\label{dfn:homogeneous}A theory $(\Theta,j)$ on $(\Ee,\Aa)$ is called \emph{homogeneous} if\begin{itemize}\item[(i)]$j$ is faithful;\item[(ii)]any morphism in $\Theta$ factors in an essentially unique way as a generic morphism followed by a free morphism;\item[(iii)]the composite of two generic morphisms is generic;\item[(iv)]the invertible morphisms of~$\Theta$ are those which are at once generic and free.\end{itemize}\end{dfn}

In other words, homogeneous theories are precisely those which contain the arities as a \emph{subcategory} and which admit a generic/free \emph{factorisation system}. For given homogeneous theories $(\Theta,j),\,(\Theta',j')$ with same arities we say that $(\Theta',j')$ is \emph{$(\Theta,j)$-homogeneous} if it comes equipped with a theory morphism $(\Theta',j')\to(\Theta,j)$ which \emph{preserves} and \emph{reflects} generic morphisms.\vspace{1ex}

It can be shown that homogeneous \emph{algebraic} theories are intimately related to \emph{symmetric operads}, the symmetries of the finite sets playing a prominent role here. We will certainly come back to this topic elsewhere, cf. Section \ref{sct:Gamma} below. In this article we are mainly concerned with the homogeneous theories associated to strongly cartesian monads, especially globular operads, cf. Sections \ref{dfn:cartesian} and \ref{dfn:omega-operads}.

\begin{thm}\label{thm:cartesian-homogeneous}Let~$\Ee$ be a finitely complete category with dense generator $\Aa$.

For any strongly cartesian monad $T$ the associated theory $(\Theta_T,j_T)$ with arities $\Aa_T$ (cf. \ref{thm:lra->arities}) is homogeneous. If the arities $\Aa_T$ have no non-trivial automorphisms, the equivalence \ref{thm:monad-theory} between monads and theories with arities $\Aa_T$ restricts to an equivalence between $T$-cartesian monads and $(\Theta_T,j_T)$-homogeneous theories.\end{thm}

\begin{proof}The functor $j_T:\Aa_T\to\Theta_T$ is faithful since it is the restricted free $T$-algebra functor, and cartesian monads are faithful. Recall that $\Theta_T$ is the subcategory of the Kleisli category $\Ee_T$ spanned by the free $T$-algebras on objects of $\Aa$. We \emph{define} the generic morphisms $g:TA\to TB$ of $\Theta_T$ to be those morphisms in $\Theta_T$ which correspond to $T$-generics $\tilde{g}:A\to TB$ (cf. \ref{dfn:generic}) under the well-known description of the Kleisli category $\Ee_T$ by such morphisms (i.e. $g=\mu_B T\tilde{g}$ with $\tilde{g}$ $T$-generic). In particular, the composite of two generic morphisms in $\Theta_T$ is again generic (cf. proof of \ref{thm:lra->arities}). Moreover, according to Lemma \ref{lma:lra}, any morphism in $\Theta_T$ factors in an essentially unique way as a generic morphism followed by a free morphism. This factorisation property guarantees that generic morphisms can only factor through free morphisms if the latter are invertible in $\Aa_T$. Finally, any isomorphism in $\Theta_T$ is generic, but also the image under $j_T$ of an isomorphism in $\Aa_T$, thus free.

Any $T$-cartesian monad $S$ has arities $\Aa_T$ and induces thus a theory $(\Theta_S,j_S)$ with arities $\Aa_T$. Because of the existence of $S$-generic factorisations (cf. proof of \ref{thm:lra->arities}) this theory is homogeneous; condition (i) is satisfied since cartesian monads are faithful. Moreover, the cartesian transformation $\phi:S\Rightarrow T$ induces a morphism of theories $\Theta_\phi:(\Theta_S,j_S)\to(\Theta_T,j_T)$ which preserves and reflects generic morphisms by \cite[5.10.2 and 5.11]{W0}, whence $(\Theta_S,j_S)$ is $(\Theta_T,j_T)$-homogeneous. Conversely, to any such theory corresponds a monad $S$ with arities $\Aa_T$ equipped with a canonical monad morphism $\phi:S\Rightarrow T$. It remains to be shown that the latter is cartesian. The way the right adjoint $M$ in the proof of \ref{thm:monad-theory} has been constructed implies that it is enough to show that the monad morphism $j_S^*(j_S)_!\Rightarrow j_T^*(j_T)_!$ is cartesian. This is a consequence of Proposition \ref{prp:homogeneous->cartesian} below.\end{proof}

\begin{prp}\label{prp:homogeneous->cartesian}For any $(\Theta_T,j_T)$-homogeneous theory $(\Theta_S,j_S)$ whose arities $\Aa_T$ admit no non-trivial automorphisms, the associated (cf. \ref{lma:theory-morphism}) monad morphism $j_S^*(j_S)_!\Rightarrow j_T^*(j_T)_!$ is cartesian.\end{prp}

\begin{proof}It suffices to prove that for any presheaf $X$ on $\Aa_T$ the unique morphism $X\to 1$ to the terminal presheaf induces a cartesian square\begin{diagram}[small]j_S^*(j_S)_!(X)&\rTo&j_T^*(j_T)_!(X)\\\dTo&&\dTo\\j_S^*(j_S)_!(1)&\rTo&j_T^*(j_T)_!(1)\end{diagram}in $\PSh{\Aa_T}$. Since $\Aa_T$ has no non-trivial automorphisms, the generic/free factorisations in $\Theta_S$ and $\Theta_T$ are \emph{unique}. Therefore, the pointwise formulae for the left Kan extensions $(j_S)_!$ and $(j_T)_!$ imply that the square above, when evaluated at an arity $A$ of $\Aa_T$, is isomorphic to the square\begin{diagram}[small]\coprod_B\coprod_{\Theta_S^{gen}(A,B)}X(B)&\rTo&\coprod_B\coprod_{\Theta_T^{gen}(A,B)}X(B)\\\dTo&&\dTo\\\coprod_B\coprod_{\Theta_S^{gen}(A,B)}1(B)&\rTo&\coprod_B\coprod_{\Theta_T^{gen}(A,B)}1(B)\end{diagram}in which $\Theta^{gen}$ denotes the subcategory of generic morphisms of $\Theta$. The horizontal arrows are induced by the same map of ``indices'' since the theory morphism $(\Theta_S,j_S)\to(\Theta_T,j_T)$ preserves and reflects generic morphisms. Therefore, for each arity $A$, the latter square is cartesian, so that the former is cartesian as well.\end{proof}

\subsection{Globular operads as $\Theta_\omega$-homogeneous theories}\label{sct:Theta}--\vspace{1ex}

A \emph{globular theory} is defined to be a theory on globular sets with arities $\Theta_0$, i.e. the canonical arities of the free $\omega$-category monad $D_\omega$, cf. Section \ref{dfn:omega-operads}. This definition of a globular theory is more restrictive than the one adopted in \cite[1.5]{Be} (resp. \cite[2.1.1]{A}), cf. Remark \ref{rmk:Paul-Andre}. However, it follows essentially from Lemma \ref{lma:useful} and Proposition \ref{prp:homogeneous->cartesian} that Definition \ref{dfn:homogeneous} of a homogeneous globular theory is equivalent to \cite[1.15]{Be} (resp. \cite[2.2.6/2.7.1]{A}).

The homogeneous globular theory associated to $D_\omega$ will be denoted by $(\Theta_\omega,j_\omega)$. The category $\Theta_\omega$ is dual to Joyal's \cite{J} category of finite combinatorial $\omega$-disks, cf. \cite[2.2]{Be}; in particular, the presheaf category $\PSh{\Theta_\omega}$ is a \emph{classifying topos} for combinatorial $\omega$-disks, cf. \cite[3.10]{Be2}. It follows from \cite[1.3]{Be} that the canonical arities $\Theta_0$ do not contain any non-trivial automorphisms. Therefore, Theorem \ref{thm:cartesian-homogeneous} implies

\begin{thm}\label{thm:globularoperads}There is a canonical equivalence between the category of Batanin's $\omega$-operads and the category of $(\Theta_\omega,j_\omega)$-homogeneous theories.\end{thm}

This equivalence can be deduced from \cite[1.16]{Be}, where however, as pointed out to us by Dimitri Ara, the augmentation over $(\Theta_\omega,j_\omega)$ has not been mentioned explicitly. Nevertheless, this augmentation is constructed in course of proving \cite[1.16(ii)$\Rightarrow$(iii)]{Be}, and used in proving \cite[1.16(iii)$\Rightarrow$(i)]{Be}. For a proof ``from scratch'' of this equivalence, we refer the reader to Ara's PhD thesis, especially \cite[6.6.8]{A}.

It is remarkable that the homogeneous $n$-globular theories $(\Theta_n,j_n)$ associated to the free $n$-category monads $D_n$ \emph{filter} $(\Theta_\omega,j_\omega)$ in a combinatorially transparent way. Indeed, using the wreath-product construction described in \cite[3.1/4]{Be2}, it is readily verified that, as well the canonical arities $\Theta_{n,0}$, as well the theories $\Theta_n$, satisfy the following recursion rule (see Section \ref{dfn:omega-operads} for notation):$$\Theta_{n+1,0}=\Theta_{1,0}\wr\Theta_{n,0}\quad\text{and}\quad\Theta_{n+1}=\Theta_1\wr\Theta_n\quad(n>0)$$ in a way that is compatible with the arity-inclusion functors $j_n:\Theta_{n,0}\to\Theta_n$, and yielding in the colimit the arity-inclusion $j_\omega:\Theta_0\to\Theta_\omega$. Since $\Theta_1$ is isomorphic to the simplex category $\Delta$, this illustrates the intricate relationship between higher categorical and higher simplicial structures.

\subsection{Symmetric operads as $\Gamma$-homogeneous theories}\label{sct:Gamma}--\vspace{1ex}

The notion of homogeneous theory has applications outside the context of cartesian monads. As illustration we discuss the homogeneous theories associated to \emph{symmetric operads} on sets. A fundamental role is played by the algebraic theory $\Theta_{\com}$ of \emph{commutative monoids} since the latter corresponds to the terminal symmetric operad. The free commutative monoid on $n$ elements is given by the $n$-th power $\NN^n$ of the additive monoid of natural numbers. The algebraic theory $(\Theta_\com,j_\com)$ is thus the full subcategory of commutative monoids spanned by these powers $\NN^n$, with evident arity-inclusion functor $j_\com$. This algebraic theory is \emph{homogeneous} in the sense of Definition \ref{dfn:homogeneous}, where a homomorphism $\NN^m\to\NN^n$ is \emph{generic} if it takes each generator of $\NN^m$ to a sum of generators of $\NN^n$ in such a way that every generator of the target appears \emph{exactly} once. Indeed, any homomorphism $\NN^m\to\NN^n$ factors in an essentially unique way as a generic followed by a free homomorphism.

Segal's category $\Gamma$ \cite{Se} can be realised as a subcategory of $\Theta_\com$, having same objects but only those homomorphisms $\NN^m\to\NN^n$ which take each generator to a sum of generators in such a way that target generators appear \emph{at most} once. The generic/free factorisation system of $\Theta_\com$ restricts to a generic/free factorisation system of $\Gamma$, in which the generic part $\Gamma_\gen$ is the same as the generic part of $\Theta_\com$, while the free part $\Gamma_0$ consists just of those free homomorphisms which are induced by \emph{injective} set mappings. This is sufficient to recover commutative monoids as those presheaves on $\Gamma$ which are sheaves on $\Gamma_0$ with respect to the evident (induced) Grothendieck topology on $\Gamma_0$. The passage from $\Theta_\com$ to $\Gamma$ can be interpreted as an elimination of those universal operations (acting on commutative monoids) which involve \emph{diagonals}. The presheaf category $\widehat{\Gamma}$ is a classifying topos for \emph{pointed objects}; in particular, Segal's category $\Gamma$ is dual to a skeleton of the category of finite pointed sets. Up to this duality, our generic/free factorisation system coincides with Lurie's active/inert factorisation system \cite{Lu}.

A \emph{$\Gamma$-homogeneous theory} in the sense of Definition \ref{dfn:homogeneous} is a pair of bijective-on-objects functors $\Gamma_0\overset{j}{\to}\Gamma_A\overset{q}{\to}\Gamma$ such that $j^*j_!$ preserves $\Gamma_0$-sheaves and such that $q$ preserves and reflects generic morphisms for a given generic/free factorisation system of $\Gamma_A$. This data determines, and is up to isomorphism uniquely determined by, a \emph{symmetric operad} $A=(A(n))_{\,n\geq 0}$ where $A(n)$ corresponds to the set of generic morphisms from $1$ to $n$ in $\Gamma_A$. In order to reconstruct the theory $\Gamma_A$ from the operad $A$ one uses the canonical isomorphisms $\Gamma_A(m,n)\cong\Gamma_A(1,n)^m$ as well as the generic/free factorisation system of $\Gamma_A$. We refer the reader to \cite{Lu} for more details, where this presentation of symmetric operads is the basis for a suitable weakening of the notion of symmetric operad itself.

Batanin's \emph{symmetrisation functors} \cite{Ba2} can also be understood along these lines. Indeed, a truncated version of Theorem \ref{thm:globularoperads} yields an equivalence between globular $n$-operads and $(\Theta_n,j_n)$-homogeneous theories. In \cite[3.3]{Be2}, a sequence of functors $\gamma_n:\Theta_n\to\Gamma$ is constructed, extending in a natural way Segal's functor $\gamma_1:\Delta\to\Gamma$, and having the property to preserve and reflect generic morphisms. In particular, pullback along $\gamma_n$ converts $\Gamma$-homogeneous theories (i.e. symmetric operads) into $\Theta_n$-homogeneous theories (i.e. globular $n$-operads). These pullback functors coincide up to isomorphism with Batanin's desymmetrisation functors. The symmetrisation functors are defined to be their left adjoints. The explicit construction of these symmetrisation functors is important for the theory of (topological) $E_n$-operads as follows from Batanin's work. Our formulation by means of homogeneous theories sheds some light onto Batanin's formula \cite[13.1]{Ba2}.

\section{The free groupoid monad.}

\subsection{Overview}
The category $\Graph$ of graphs, which is the category of presheaves on the category $\GraphSite$
\[ \xygraph{*=(0,0){\xybox{\xygraph{{0}="l" [r] {1}="r" "l":@<1ex>"r"^-{\sigma} "l":@<-1ex>"r"_-{\tau}}}}} \]
admits an involution $(-)^{\op}:\Graph \to \Graph$. Given a graph $X$, $X^{\op}$ has the same vertices as $X$, and an edge $a \to b$ in $X^{\op}$ is by definition an edge $b \to a$ in $X$. An \emph{involutive graph} is a pair $(X,\iota)$ consisting of a graph $X$, together with a graph morphism $\iota:X^{\op} \to X$ called the involution, which satisfies $\iota^{\op}\iota=1_X$. Involutive graphs form a category $\InvGraph$, with a map $f:(X,\iota_1) \to (Y,\iota_2)$ being a graph morphism $f$ such that $f\iota_1=\iota_2f^{\op}$. In terms of sets and functions an involutive graph amounts to: sets $X_0$ (of vertices) and $X_1$ (of edges), and functions $s,t:X_1 \to X_0$ (the source and target) and $\iota:X_1 \to X_1$ (the involution) such that $\iota^2=1_{X_1}$ and $s\iota=t$, and so $\InvGraph$ is the category of presheaves on the category $\InvGraphSite$
\[ \xygraph{*=(0,0){\xybox{\xygraph{{0}="l" [r] {1}="r" ([l(2)u(1.5)] {}="c1", [l(2)d(1.5)] {}="c2") "r":@( {"c1";"r"}, {"c2";"r"})"r"_{\iota} "l":@<1ex>"r"^-{\sigma} "l":@<-1ex>"r"_-{\tau}}}}} \]
in which $\iota^2=\id$ and $\iota\sigma=\tau$. Given an edge $f:a \to b$ in an involutive graph, we shall refer to the edge $\iota(f):b \to a$ as the \emph{dual} of $f$.

Left Kan extension and restriction along the identity-on-objects inclusion $k:\GraphSite \to \InvGraphSite$ can be described as follows. Restriction $k^*$ is the forgetful functor $\InvGraph \to \Graph$. Its left adjoint $k_!$ can be obtained on objects as a pushout
\[ \xygraph{{Z_0}="tl" [r] {Z^{\op}}="tr" [d] {k_!Z}="br" [l] {Z}="bl" "tl":"tr"^-{}:"br"^-{}:@{<-}"bl"^-{}:@{<-}"tl"^-{} "br" [u(.3)l(.15)] :@{-}[l(.15)]:@{-}[d(.15)]} \]
(over the set of its vertices) of the original graph $Z$ and its dual. In particular for each natural number $n$ the graph $k_![n]$ has object set $\{0,...,n\}$ and a unique edge $i \to j$ when $|i-j|=1$. For instance the involutive graph $k_![3]$ looks like this:
\[ \xygraph{{0}="p1" [r] {1}="p2" [r] {2}="p3" [r] {3.}="p4" "p1":@<1ex>"p2":@<1ex>"p3":@<1ex>"p4":@<1ex>"p3":@<1ex>"p2":@<1ex>"p1"} \]
We shall denote by $\ca Seq$ the full subcategory of $\InvGraph$ consisting of the \emph{finite sequences}, that is to say, the involutive graphs of the form $k_![n]$.

More generally any graph in the combinatorialists' sense, namely a set $X$ equipped with a symmetric relation $R$, has an associated involutive graph: the vertices are the elements of $X$ and there is a unique edge $a \to b$ iff $(a,b) \in R$. In particular we shall consider the full subcategory $\Acyc$ of $\InvGraph$ consisting of the \emph{finite connected acyclic graphs}. Of course $\Acyc$ contains $\ca Seq$. However the following figures
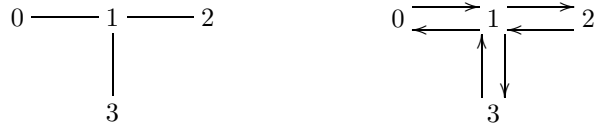
\[ \xygraph{{\xybox{\xygraph{{0}="l" [r] {1}="m" [r] {2}="r" [dl] {3}="b" "l":@{-}"m"^-{}(:@{-}"r"^-{},:@{-}"b")}}}
[r(4)] {\xybox{\xygraph{{0}="l" [r] {1}="m" [r] {2}="r" [dl] {3}="b" "l":@<1ex>"m"^-{}:@<1ex>"r"^-{} "r":@<1ex>"m"^-{}:@<1ex>"l"^-{} "m":@<1ex>"b"^-{}:@<1ex>"m"^-{}}}}} \]
describe an object of $\Acyc$ (represented on the left as a combinatorial graph and on the right as an object of $\InvGraph$) which is not isomorphic in $\InvGraph$ to any finite sequence.

Every \emph{groupoid} has an underlying involutive graph. Its vertices and edges are the objects and morphisms of the groupoid, and the involution is given by $f \mapsto f^{-1}$. This is the object part of a monadic forgetful functor $\Gpd \to \InvGraph$ from the category of groupoids to that of involutive graphs, and we shall denote the corresponding monad -- the \emph{free groupoid monad} --  by $G$. The purpose of this section is to describe arities for $G$, from which an application of the nerve theorem \ref{T2} recovers the basic aspects of the symmetric simplicial nerve of a groupoid. We begin our analysis of $G$ by pointing out that it is not cartesian.

\begin{prp}
The functor $G:\InvGraph \to \InvGraph$ does not preserve pullbacks.
\end{prp}
\begin{proof}
The terminal involutive graph $1$ has one vertex and one edge which is necessarily its own dual. Since a morphism of involutive graphs $1 \to X$, where $X$ is a groupoid, is the same thing as an involution in $X$, it follows that $G1$ is $\Z_2$ regarded as a one object groupoid. Denote by $E$ the involutive graph with one vertex and two edges which are dual to each other. Then a morphism of involutive graphs $E \to X$, where $X$ is a groupoid, is the same thing as an arbitrary endomorphism in $X$, and so $GE$ is the group $\Z$ regarded as a one object groupoid. Write $P$ for the involutive graph obtained as the kernel pair of the unique map $E \to 1$, that is, $P = E \times E$. Thus $P$ has one vertex and 4 edges -- two pairs of dual edges. Since a morphism $P \to X$ for $X$ a groupoid amounts to a pair of endomorphisms on the same object in $X$, $GP$ is the free group on 2 generators regarded as a one object groupoid. To say that the pullback defining $P$ is preserved by $G$ is equivalent to saying that the induced square
\[ \xygraph{{GP}="tl" [r] {\Z}="tr" [d] {\Z_2}="br" [l] {\Z}="bl" "tl":"tr"^-{}:"br"^-{}:@{<-}"bl"^-{}:@{<-}"tl"^-{}} \]
of groups is a pullback in the category of groups. But this cannot be since a pullback (in fact any limit) of abelian groups is itself abelian, whereas $GP$ is not abelian.
\end{proof}

\subsection{The monad for involutive categories}\label{sec:iCat-monad}
Our quest to understand the monad $G$ begins with an understanding of a closely related monad $T$ on $\InvGraph$. The algebras of $T$ are \emph{involutive categories} (aka \emph{dagger categories}). As we shall see, the monad $T$ is strongly cartesian, its canonical arities are given by $\ca Seq$, and $T$ possesses a system of idempotents (related to the fact that groupoids form an epireflective subcategory of $\InvGraph^T$) which will enable us to obtain $G$ from $T$.

Let us denote by $(D_1,\mu,\eta)$ the monad on $\Graph$ whose algebras are categories. The graph $D_1X$ has the same vertices as $X$, and an edge $a \to b$ in $D_1X$ is a path $a \to b$ in $X$. The unit $\eta_X:X \to D_1X$ picks out the paths of length $1$, and the multiplication $\mu_X:D_1^2X \to D_1X$ is described by the concatenation of paths. The free category monad $D_1$ on graphs \emph{lifts} to the category of involutive graphs, thereby inducing the free involutive category monad $T$. One defines on objects $T(X,\iota) = (D_1X,D_1\iota)$, and the arrow map of $T$, as well as the unit and the multiplication of $T$, are inherited in an evident way from $D_1$. More abstractly, there is a canonical \emph{distributive law} (in the sense of Beck) $k^*k_!D_1\Rightarrow D_1k^*k_!$ which induces the aforementioned lifting of $D_1$ to $\Graph^{k^*k_!}=\InvGraph$. We shall now explain how the canonical arities of $T$ are obtained from those of $D_1$.
\begin{prp}\label{prop:free-arities}Let $k:\CC\to\DD$ be a bijective-on-objects functor, and let $S$ and $T$ be a monads on~$\PSh{\CC}$ and~$\PSh{\DD}$ such that $k^*T=Sk^*$ fulfilling identities (\ref{dfn:monad-morphism}). If $S$ is strongly cartesian, then $T$ is strongly cartesian, and the canonical arities for $T$ are obtained from those for $S$ by left Kan extension along $k$ (cf. \ref{thm:lra->arities} and \ref{rmk:canonical arities}).\end{prp}

\begin{proof}By Freyd's adjoint functor theorem, a functor between presheaf categories is a local right adjoint if and only if it preserves connected limits. Note that since $k$ is bijective on objects, $k^*$ is monadic and so creates all limits. Thus since $k^*T=Sk^*$ and $k^*$ creates connected limits, $T$ preserves them since $S$ does, and so the endofunctor $T$ is also a local right adjoint. Moreover the naturality squares of $T$'s unit and multiplication are cartesian, since: (1) those of $S$ are, (2) the monad $T$ is a lifting of the monad $S$, and (3) $k^*$ reflects pullbacks.

The canonical arities (cf. \ref{rmk:canonical arities}) for $S$ and $T$ appear as the essential image of the top and bottom horizontal composite functors in
\[ \xygraph{!{0;(3,0):(0,.5)::}
{y_{\CC}/S1}="tl" [r(.8)] {\PSh{y_{\CC}/S1}}="tm1" [r(.2)] {\catequiv}="tm2" [r(.2)] {\PSh{\CC}/S1}="tm3" [r(.8)] {\PSh{\CC}}="tr" [d] {\PSh{\DD}}="br" [l(.8)] {\PSh{\DD}/T1}="bm3" [l(.2)] {\catequiv}="bm2" [l(.2)] {\PSh{y_{\DD}/T1}}="bm1" [l(.8)] {y_{\DD}/T1}="bl"
"tl":"tm1"^-{\tn{yoneda}} "tm3":"tr"^-{L^S_1}:"br"^-{k_!}:@{<-}"bm3"^-{L^T_1} "bm1":@{<-}"bl"^-{\tn{yoneda}}:@{<-}"tl"^-{} "tm1":"bm1" "tm3":"bm3"
"tl":@{}"bm1"|{\iso} "tm1":@{}"bm3"|{=} "tm3":@{}"br"|{\iso}
} \]
where $L^S_1 \ladj S_1$ and $L^T_1 \ladj T_1$.  The vertical arrows are all induced by left Kan extension $k_!:\PSh{\CC}\to\PSh{\DD}$, where we use that $S1=Sk^*1=k^*T1$ so that we have a canonical morphism $k_!S1=k_!k^*T1\to T1$. It follows then that the left square pseudo-commutes because left Kan extensions preserve representables, the middle square commutes on the nose, and the right square pseudo-commutes by taking left adjoints of the functors participating in the equation $k^*_{T1}T_1=S_1k^*$.\end{proof}
\noindent Applying Proposition \ref{prop:free-arities} to the inclusion $k:\GraphSite \to \InvGraphSite$ gives
\begin{cor}\label{cor:S-ar-4-T}The involutive category monad $T$ has canonical arities $\ca Seq$.\end{cor}
\begin{rmk}\label{rmk:otherarities}There are other arities for $T$ that are worth considering, in particular we shall see now that $\Acyc$ also endows $T$ with arities. For this we show first that a $T$-generic morphism $g:B \to TA$ exhibits the involutive graph $A$ as a refinement of $B$, obtained by subdividing the edges of $B$ into paths, cf. \cite[2.5]{W}. Formally one has the adjunction
\[ \xygraph{{\InvGraph}="l" [r(3)] {\InvGraph/T1}="r" "l":@<-1.2ex>"r"_-{T_1}|-{}="b":@<-1.2ex>"l"_-{L_1}|-{}="t" "t":@{}"b"|{\perp}} \]
and $g$ is recovered by setting $f=T(t_A)g$ and then looking at the component of the unit of this adjunction at $f$.  Thus the following more explicit description of this adjunction is useful. Note that $T1$ has one object, an edge for each $n \in \N$, and its involution is the identity. An object of $\InvGraph/T1$ may be regarded as an involutive graph whose edges are labelled by natural numbers, such that dual edges have the same label. In this way a morphism of $\InvGraph/T1$ is a label-preserving morphism of involutive graphs. The functor $L_1$ sends a given labelled involutive graph $X$ to the involutive graph $Y$ obtained from $X$ by replacing each dual pair of edges labelled by $n$ by a path of length $n$, that is to say, by a copy of $k_![n]$. When $n$ is zero this amounts to identifying the source and target, when $n=1$ this results in no change, and for larger $n$ one must add new intermediate vertices. The unit of $L_1 \ladj T_1$ is the labelled involutive graph morphism $X \to TY$ which sends each dual pair of edges of $X$ to the path in $Y$ which replaces it in the construction of $Y$. In general a morphism $g:B \to TA$ amounts to a function $g_0:B_0 \to A_0$, and an assignation of a path $g_0b_1 \to g_0b_2$ in $A$ to each edge $b_1 \to b_2$ in $B$, this assignation being compatible with duals. This morphism is $T$-generic, by the above explicit description of the adjunction $L_1 \ladj T_1$, if and only if each edge in $A$ appears exactly once as part of a path picked out by $g$.\end{rmk}

By \cite[3.9]{GU}, any small full subcategory $\ca B$ of $\InvGraph$ containing the representables (i.e. $k_![0]$ and $k_![1]$) is dense. Recall from \ref{dfn:cartesian} that such a $\ca B$ is said to be \emph{$T$-generically closed} if given $p \in \ca B$ and a $T$-generic $p\to Tq$ (i.e. $q$ is obtained from $p$ by replacing some of the dual pairs of edges of $p$ by paths), then $q$ is also in $\ca B$. As in the proof of Theorem \ref{thm:lra->arities}, any $T$-generically closed dense generator of $\InvGraph$ endows $T$ with arities. In particular, since $\Acyc$ is $T$-generically closed, we get
\begin{cor}\label{cor:A-arities-4-T}The free involutive category monad $T$ has arities $\Acyc$.\end{cor}
\begin{rmk}\label{rem:acyclic-substitution}
Let $p$ be a finite connected acyclic graph, $a$ and $b$ vertices of $p$ connected by an edge, $q$ another finite connected acyclic graph, and $c$ and $d$ vertices of $q$. Then the graph $r$ obtained by identifying $a$ with $c$ and $b$ with $d$, removing the edge from $a$ to $b$ in $p$, and retaining the rest of $p$ and $q$ unchanged yet distinct, is also finite connected acyclic. In particular when the graphs in question are finite sequences, this is just the usual path substitution. When $p$ is general but $q$ is a finite sequence, this is the combinatorial observation responsible for \ref{cor:A-arities-4-T}.\end{rmk}

\subsection{System of idempotents for $T$ and reduced paths}\label{sec:system-of-idempotents}The category $\Gpd$ of groupoids is an \emph{epireflective} subcategory of the category $\InvGraph^T$ of involutive categories. This means that the inclusion $\Gpd\inc\InvGraph^T$ has a left adjoint reflection $\InvGraph^T\to\Gpd$ such that the unit of the adjunction is a pointwise epimorphism. This unit induces, for each involutive graph $X$, an epimorphism $r_X:TX\to GX$ which identifies $GX$ with a quotient of $TX$. The quotient map is obtained by identifying two paths if they have the same associated \emph{reduced path}.

More precisely, the edges of $TX$ are the paths in $X$. A \emph{redundancy} in a path is a subpath of length 2 of the form
\[ \xygraph{{a}="l" [r] {b}="m" [r] {a}="r" "l":"m"^-{f}:"r"^-{\iota(f)}} \]
where $f$ is an edge in $X$. Given any path one can remove all redundancies, and this is the edge map of an identity-on-objects morphism $\tau_X:TX \to TX$. These morphisms have the following easily verified properties:
\begin{enumerate}
\item (idempotent): for all $X$, $\tau_X^2=\tau_X$.\label{gpdmnd:idempotent}
\item (weak naturality): for all $f:X \to Y$, $\tau_YT(f)\tau_X = \tau_YT(f)$.\label{gpdmnd:wk-naturality}
\item (multiplication): for all $X$, $\tau_X\mu_X=\tau_X\mu_X\tau_{TX}T(\tau_X)$.\label{gpdmnd:mult}
\item (unit): for all $X$, $\tau_X\eta_X = \eta_X$.\label{gpdmnd:unit}
\end{enumerate}
A $T$-algebra $(X,\xi_X)$ is a groupoid if and only if for all edges $f:a \to b$ in $X$, $\iota(f)f = 1_a$. Thus $\Gpd$ is the full subcategory of $\InvGraph^T$ consisting of those $T$-algebras $(X,\xi_X)$ which satisfy $\xi_X\tau_X = \xi_X$. For an involutive graph $X$, the edges of $TX$ which are $\tau_X$-invariant are called \emph{reduced paths} in $X$, i.e. a path in $X$ is reduced precisely when it contains \emph{no redundancies}. A path in $X$ can also be seen as a morphism $p:k_![1] \to TX$ which is reduced precisely when $\tau_Xp=p$. More generally then, we shall say that a morphism $p:B \to TX$ is reduced when $\tau_Xp=p$; in particular, $p$ sends each edge of $B$ to a reduced path in $X$.

The reason for which we insist on the aforementioned description of a \emph{system of idempotents} $\tau$ associated to $T$ is that the free groupoid monad $G$ is entirely recoverable from $(T,\tau)$. Indeed, $G$ is obtained from $(T,\tau)$ by \emph{splitting idempotents}, which is to say that for each involutive graph $X$ we \emph{choose} graph morphisms\[ \begin{array}{lccr} {r_X : TX \to GX} &&& {i_X : GX \to TX} \end{array} \]such that $r_Xi_X = 1_{GX}$ and $i_Xr_X=\tau_X$. For each morphism $f:X \to Y$ in $\InvGraph$, we \emph{define} $Gf$ to be $r_YT(f)i_X$, and this is the arrow map of an endofunctor. The morphisms $r_X:TX \to GX$ are then the components of a natural transformation. Note that the $i_X$ are \emph{not} natural in $X$. The unit of $G$ is defined as
\[ \xygraph{{X}="l" [r] {TX}="m" [r] {GX}="r" "l":"m"^-{\eta_X}:"r"^-{r_X}} \]
and the multiplication of $G$ is defined as
\[ \xygraph{!{0;(1.5,0):(0,1)::} {G^2X}="p1" [r] {GTX}="p2" [r] {T^2X}="p3" [r] {TX}="p4" [r] {GX.}="p5" "p1":"p2"^-{Gi_X}:"p3"^-{i_{TX}}:"p4"^-{\mu_X}:"p5"^-{r_X}}  \]
It is then readily verified that\begin{prp}The endofunctor $G$ together with unit and multiplication just described is the free groupoid monad on $\InvGraph$, and $r:T \to G$ is the monad morphism induced by the unit of the adjunction between $G$-algebras and $T$-algebras.
\end{prp}

\begin{prp}\label{prp:S-not-arities-for-G}The free groupoid monad $G$ is not a monad with arities $\ca Seq$.\end{prp}
\begin{proof}Let $f:k_![1] \to GX$ be a graph morphism which picks out a reduced path\begin{equation}\label{eq:seq-for-counterex}\xygraph{{x_0}="l" [r] {x_1}="m" [r] {x_2}="r" "l":"m"^-{f_0}:"r"^-{f_1}}\end{equation}in $X$. An object $(g,k_![m],h)$ in $\Fact{\ca Seq}{G}{f}$ is by definition a factorisation of $f$ of the form\[ \xygraph{!{0;(1.5,0):} {k_![1]}="l" [r] {Gk_![m]}="m" [r] {GX.}="r" "l":"m"^-{g}:"r"^-{Gh}} \]But $Gk_![m]$ is the chaotic category on the set $\{0,...,m\}$, so an edge $g:k_![1] \to Gk_![m]$ is determined uniquely by its end points. Thus the data $(g,k_![m],h)$ amounts to a path $h:k_![m] \to X$ and an ordered pair $(i,j)$ from $\{0,...,m\}$, such that the segment $i \to j$ of the path $h$ is (\ref{eq:seq-for-counterex}) once redundancies have been removed. Let us denote this data as $(i,j,h)$. A morphism $\delta:(i_1,j_1,h_1) \to (i_2,j_2,h_2)$ in $\Fact{\ca Seq}{G}{f}$ amounts to a morphism\[ \xygraph{{k_![m_1]}="l" [r(2)] {k_![m_2]}="r" [dl] {X}="b" "l":"r"^-{\delta}:"b"^-{h_2}:@{<-}"l"^-{h_1}} \]in $\InvGraph$ over $X$ such that $\delta i_1=i_2$ and $\delta j_1=j_2$. An \emph{inner redundancy} for $(i,j,h)$ is a subpath of length 2 of $h$, contained in the segment $i \to j$, of the form
\[ \xygraph{{x_2}="l" [r] {y}="m" [r] {x_2.}="r" "l":"m"^-{s}:"r"^-{ds}} \]
Note that given a morphism $\delta:(i_1,j_1,h_1) \to (i_2,j_2,h_2)$ and an inner redundancy $r$ for $(i_1,j_1,h_1)$, one gets an inner redundancy for $(i_2,j_2,h_2)$ by applying $\delta$ to the length 2 segment of $k_![m_1]$ whose image is $r$. Conversely, note that $\delta$ restricted to the segment $i_1 \to j_1$ maps onto the segment $i_2 \to j_2$, thus if $(i_2,j_2,h_2)$ has an inner redundancy, then so does $(i_2,j_2,h_2)$. Thus if $(i,j,h)$ has an inner redundancy, then so do all the other objects of $\Fact{\ca Seq}{G}{f}$ in its connected component. Suppose $X$ has a vertex $y$ distinct from the $x_i$ and an edge $s:x_1 \to y$. Take $h_1$ to be the path $(f_0,f_1)$ which has no redundancies, $i_1=0$ and $j_1=2$. Take $h_2$ to be the path $(f_0,s,ds,f_1)$, $i_2=0$ and $j_2=4$. Then $(i_1,j_1,h_1)$ doesn't have an inner redundancy whereas $(i_2,j_2,h_2)$ does, and so they are in different components of $\Fact{\ca Seq}{G}{f}$. Thus by Proposition \ref{prp:monad-with-arities-elementary} the result follows.\end{proof}
Sequences are thus not enough to give $G$ arities because of certain inner redundancies. To overcome this we consider instead finite acyclic connected graphs. First we isolate the analogue of generic factorisations for $G$. A morphism $g:B \to GX$ is said to be \emph{$G$-generic} when it factors as\[ \xygraph{!{0;(1.5,0):} {B}="l" [r] {TX}="m" [r] {GX}="r" "l":"m"^-{\tilde{g}}:"r"^-{r_X}} \]where $\tilde{g}$ is \emph{$T$-generic} in the sense of Section \ref{dfn:generic} and \emph{reduced} in the sense of Section \ref{sec:system-of-idempotents}; in particular, we have $\tilde{g}=\tau_X\tilde{g}=i_Xg$.
\begin{lma}\label{lem:G-genfact}Every $f:B \to GX$ may be factored as\[ \xygraph{!{0;(1.5,0):} {B}="l" [r] {GA}="m" [r] {GX}="r" "l":"m"^-{g}:"r"^-{Gh}} \]where $g$ is $G$-generic. Moreover if $B$ is a sequence (resp. a finite connected acyclic graph), then so is $A$.\end{lma}
\begin{proof}
In the diagram
\[ \xygraph{ !{0;(1,0):(0,1.5)::} {B}="tl" [r(2)] {GX}="tm" [r(2)] {TX}="tr" [dl] {TA}="br" [l(2)] {GA}="bl" "tl":"tm"^-{f}:@<-1ex>"tr"_-{i_X}:@{<-}"br"^-{Th}:@<-1ex>"bl"_-{r_A}:@{<-}"tl"^-{g} "tr":@<-1ex>"tm"_-{r_X} "bl":"tm"_(.7){Gh} "tl":"br"^(.75){g'} "bl":@<-1ex>"br"_-{i_A}} \]
we first generically factor $i_Xf=T(h)g'$ using the fact that $T$ is strongly cartesian and then put $g=r_Ag'$. It follows that $f=G(h)g$ by the naturality of $r$ and $r_Xi_X=1_{GX}$. Observe that (by $T$-genericity) $g'$ is reduced, since $i_Xf$ is reduced. The second assertion follows from the fact that sequences and finite connected acyclic graphs are $T$-generically closed as explained in Remark \ref{rmk:otherarities}.\end{proof}
The following lemma is the main technical result of this section; it explains in which sense $T$-generics between finite acyclic graphs are compatible with the process of removing redundancies.
\begin{lma}\label{lem:tau-compat}
Suppose that $g_1$, $g_2$, $h_1$ and $h_2$ as in
\begin{equation}\label{diag:tau-compat-lemma}
\xygraph{!{0;(1.5,0):(0,.667)::} {p}="tl" [r(2)] {Ts}="tr" [d(2)] {TX}="br" [l(2)] {Tq}="bl" [ur] {Tt}="m" "tl":"tr"^-{g_2}:"br"^-{Th_2}:@{<-}"bl"^-{Th_1}:@{<-}"tl"^-{g_1} "m"(:@{<.}"tl"|-{g_3},:@{<.}"tr"|-{T\delta_2},:@{<.}"bl"|-{T\delta_1},:@{.>}"br"|-{Th_3})}
\end{equation}
are given such that $g_1$ and $g_2$ are $T$-generic, $\tau_XT(h_1)g_1=T(h_2)g_2$ and $p \in \Acyc$, then there exists $t \in \Acyc$, $g_3$, $h_3$, $\delta_1$ and $\delta_2$ as indicated in (\ref{diag:tau-compat-lemma}), such that
\[ \begin{array}{lcccr} {\tau_tT(\delta_1)g_1 = g_3 = T(\delta_2)g_2} && {h_3\delta_1 = h_1} && {h_3\delta_2 = h_2.} \end{array} \]
\end{lma}
\begin{proof}
Observe that the equation $\tau_XT(h_1)g_1=T(h_2)g_2$ guarantees that the composite $T(h_2)g_2$ is reduced. Let us consider first a special case. We assume $p=k_![1]$ so that $q$ and $s$ are sequences also, so we write $q=k_![m]$ and $s=k_![n]$, and then we assume $g_1$ (resp. $g_2$) picks out the unique reduced path from $0$ to $m$ (resp. $n$). Thus in this case we have paths
\[ \begin{array}{lccr} {h_1:k_![m] \to X} &&& {h_2:k_![n] \to X} \end{array} \]
in $X$, $h_2$ is reduced, and $h_1$ differs from $h_2$ only in that it may contain some redundancies, and so $m \geq n$.

Let us organise in more detail these redundancies making up the difference between $h_1$ and $h_2$. We define a \emph{general redundancy} on a vertex $y$ in $Y \in \InvGraph$, to be a path $p$ from $y$ to itself whose associated reduced path is empty, that is to say, $\tau_Y(p)$ is the empty path. For $k \in \N$ we define a \emph{basic $k$-redundancy} on $y$ to be a path of the form $(e_1,...,e_k,de_k,...,de_1)$, the case $k=1$ being what we called a mere redundancy in Section \ref{sec:system-of-idempotents}. We define an \emph{irreducible $k$-redundancy} to be a basic $k$-redundancy that is not decomposable into a sequence of basic redundancies of order less than $k$. By straight forward inductive arguments one may verify that any general redundancy on a vertex $y$ in $Y \in \InvGraph$ can be written as a composite of basic redundancies, and in fact uniquely as a composite of irreducible redundancies.

In our case $h_1$ may thus be regarded as consisting of $h_2$ together with a finite sequence of basic redundancies at each vertex of the path $h_2$. Here is an example to illustrate. Consider a diagram in $X$
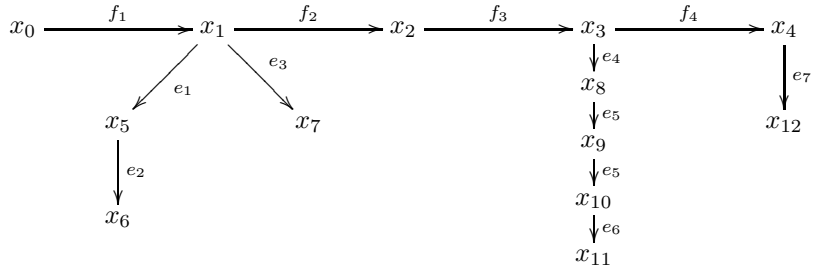
\[ \xygraph{{x_0}="p1" [r(2)] {x_1}="p2"([dl] {x_5}="p2l1" [d] {x_6}="p2l2",[dr] {x_7}="p2r", [r(2)] {x_2}="p3" [r(2)] {x_3}="p4" ([d(.6)]{x_8}="p41" [d(.6)] {x_9}="p42" [d(.6)] {x_{10}}="p43" [d(.6)] {x_{11}}="p44", [r(2)] {x_4}="p5" [d] {x_{12}}="p51")) "p1":"p2"^-{f_1}:"p3"^-{f_2}:"p4"^-{f_3}:"p5"^-{f_4} "p2"(:"p2l1"^-{e_1}:"p2l2"^-{e_2},:"p2r"^-{e_3}) "p4":"p41"^-{e_4}:"p42"^-{e_5}:"p43"^-{e_5}:"p44"^-{e_6} "p5":"p51"^-{e_7}} \]
in which the horizontal path $(f_1,f_2,f_3,f_4)$ is reduced. Take $h_2$ to be this path. Then one could take $h_1$ to be the path
\[ (f_1,e_1,e_2,de_2,de_1,e_3,de_3,f_2,f_3,e_4,e_5,e_6,e_7,de_7,de_6,de_5,de_4,f_4,e_7,de_7). \]
So in this example one has the empty sequence of basic redundancies at the vertices $x_0$ and $x_2$, the sequence $((e_1,e_2,de_2,de_1),(e_3,de_3))$ at $x_1$, etc. From the data of a general path $h_1$ and some decomposition of its redundancies into basic ones, one can construct a finite connected acyclic graph by taking first the sequence associated to its underlying reduced path, and at each vertex splicing in a path of length $k$, starting from this vertex, for each basic $k$ redundancy appearing in the given decomposition of its general redundancies. For the above illustrative example this is of course
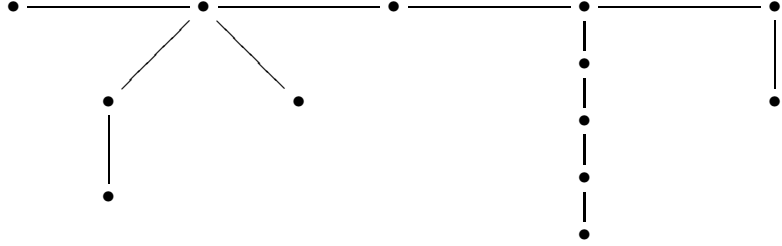
\[ \xygraph{{\bullet}="p1" [r(2)] {\bullet}="p2"([dl] {\bullet}="p2l1" [d] {\bullet}="p2l2",[dr] {\bullet}="p2r", [r(2)] {\bullet}="p3" [r(2)] {\bullet}="p4" ([d(.6)] {\bullet}="p41" [d(.6)] {\bullet}="p42" [d(.6)] {\bullet}="p43" [d(.6)] {\bullet}="p44", [r(2)] {\bullet}="p5" [d] {\bullet}="p51")) "p1":@{-}"p2"^-{}:@{-}"p3"^-{}:@{-}"p4"^-{}:@{-}"p5"^-{} "p2"(:@{-}"p2l1"^-{}:@{-}"p2l2"^-{},:@{-}"p2r"^-{}) "p4":@{-}"p41"^-{}:@{-}"p42"^-{}:@{-}"p43"^-{}:@{-}"p44"^-{} "p5":@{-}"p51"^-{}} \]
For general $h_1$ and $h_2$ we call (the associated involutive graph of) this finite connected acyclic graph $t$. The morphism $g_3:k_![1] \to Tt$ picks out the horizontal reduced path, $\delta_1:[m] \to t$ is the path that travels along the horizontal but also visits each basic redundancy as it arises in $t$, $\delta_2:[n] \to t$ is the path that just travels in the horizontal direction from left to right, and $h_3$ sends $t$ to the image of $h_1$. Clearly this data satisfies the axioms demanded by the statement of this result, and so we have proved this result in the special case $p=k_![1]$, and $g_1$ and $g_2$ as described above.

Obtaining the general case one uses Remark \ref{rem:acyclic-substitution} and the above special case. For given a general $p$ now, and an edge $e$ in $p$, restricting $g_1$, $g_2$, $h_1$, $h_2$ to the image of $e$ in $Tq$, $Ts$ and $TX$, gives an instance of our special case. Thus one constructs the associated finite connected acyclic graph $t_e$, and the associated data $g_{3,e}$, $h_{3,e}$, $\delta_{1,e}$ and $\delta_{2,e}$. The graph $t$ is then obtained by starting with the graph $p$, and at each edge $e$ substituting in the graph $t_e$ following Remark \ref{rem:acyclic-substitution}, and so $t$ is also a finite connected acyclic graph. By construction the data $g_3$, $h_3$, $\delta_1$ and $\delta_2$ is defined uniquely so that its restriction to each $e$ in $p$ is the data $g_{3,e}$, $h_{3,e}$, $\delta_{1,e}$ and $\delta_{2,e}$.
\end{proof}
The important implication of the last lemma at the level of the monad $G$ is
\begin{lma}\label{lem:G-gen-connected}
Let $f:p \to GX$, $p \in \Acyc$, $(g_1,q_1,h_1)$ and $(g_2,q_2,h_2)$ be in $\Fact{\Acyc}{G}{f}$, and $g_1$ and $g_2$ be $G$-generic. Then $(g_1,q_1,h_1)$ and $(g_2,q_2,h_2)$ are in the same connected component of $\Fact{\Acyc}{G}{f}$.
\end{lma}
\begin{proof}
The top inner square of
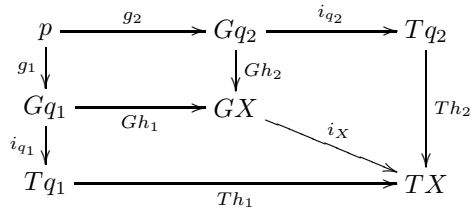
\[ \xygraph{!{0;(2,0):(0,.4)::} {p}="tl" [r] {Gq_2}="tr" [d] {GX}="br" [l] {Gq_1}="bl" [d] {Tq_1}="ol" [r(2)] {TX}="om" [u(2)] {Tq_2}="or" "tl":"tr"^-{g_2}:"br"^-{Gh_2}:@{<-}"bl"^-{Gh_1}:@{<-}"tl"^-{g_1} "or":"om"^-{Th_2}:@{<-}"ol"^-{Th_1} "bl":"ol"_-{i_{q_1}} "br":"om"^-{i_X} "tr":"or"^-{i_{q_2}}} \]
commutes by definition, but the outer diagram only commutes after post-composition with $\tau_X$ (recall that $i_Z$ is not natural in $Z$). Take a generic factorisation
\[ \xygraph{!{0;(1.5,0):} {p}="l" [r] {Tq}="m" [r] {TX}="r" "l":"m"^-{g}:"r"^-{Th}} \]
of the composite $\tau_XT(h_1)i_{q_1}g_1=\tau_XT(h_2)i_{q_2}g_2$. Now we apply Lemma \ref{lem:tau-compat} twice to produce
\[ \xygraph{{\xybox{\xygraph{!{0;(1.2,0):(0,.667)::} {p}="tl" [r(2)] {Tq}="tr" [d(2)] {TX}="br" [l(2)] {Tq_1}="bl" [ur] {Tt_1}="m" "tl":"tr"^-{g}:"br"^-{Th}:@{<-}"bl"^-{Th_1}:@{<-}"tl"^-{i_{q_1}g_1} "m"(:@{<.}"tl"|-{g_3},:@{<.}"tr"|-{T\delta_2},:@{<.}"bl"|-{T\delta_1},:@{.>}"br"|-{Th_3})}}} [r(4.5)]
{\xybox{\xygraph{!{0;(1.2,0):(0,.667)::} {p}="tl" [r(2)] {Tq}="tr" [d(2)] {TX}="br" [l(2)] {Tq_2}="bl" [ur] {Tt_2}="m" "tl":"tr"^-{g}:"br"^-{Th}:@{<-}"bl"^-{Th_2}:@{<-}"tl"^-{i_{q_2}g_2} "m"(:@{<.}"tl"|-{g_4},:@{<.}"tr"|-{T\delta_4},:@{<.}"bl"|-{T\delta_3},:@{.>}"br"|-{Th_4})}}}} \]
and then because of the equations satisfied by this data from Lemma \ref{lem:tau-compat}, one verifies that $\delta_1$,  $\delta_2$, $\delta_3$ and $\delta_4$ are well-defined morphisms
\[ \xygraph{!{0;(2,0):} {(g_1,q_1,h_1)}="p1" [r] {(r_{t_1}g_3,t_1,h_3)}="p2" [r] {(r_qg,q,h)}="p3" [r] {(r_{t_2}g_4,t_2,h_4)}="p4" [r] {(g_2,q_2,h_2)}="p5" "p1":"p2"^-{\delta_1}:@{<-}"p3"^-{\delta_2}:"p4"^-{\delta_4}:@{<-}"p5"^-{\delta_3}} \]
in $\Fact{\Acyc}{G}{f}$ as indicated in this last display.
\end{proof}
From Lemma \ref{lem:G-genfact}, Lemma \ref{lem:G-gen-connected} and Proposition \ref{prp:monad-with-arities-elementary} we obtain
\begin{thm}\label{thm:A-arities-for-G}
The free groupoid monad $G$ has arities $\Acyc$.
\end{thm}
By definition, the theory $(\Theta_G,j_G)$ associated to the monad $G$ with arities $\Acyc$ has as objects the finite connected acyclic graphs, and as morphisms the functors between the free groupoids on such graphs. Since the category $\ca Seq$ is a full subcategory of $\Acyc$, this theory restricts to a theory $(\widetilde{\Theta}_1,\tilde{j}_1)$ whose objects are natural numbers, and whose morphisms $n \to m$ are functors $Gk_![n] \to Gk_![m]$. But $Gk_![n]$ is the chaotic category on the set $\{0,...,n\}$, i.e. a functor $Gk_![n] \to Gk_![m]$ is uniquely determined by its object map. Thus $\widetilde{\Theta}_1$ is the category of non-empty finite sets and set mappings, which is sometimes denoted $\Delta_\sym$. Grothendieck's symmetric simplicial nerve characterisation \cite{G} of a groupoid would follow if we could apply the nerve theorem to $(\widetilde{\Theta}_1,\tilde{j}_1)$. However Proposition \ref{prp:S-not-arities-for-G} says that we cannot since $\ca Seq$ does not endow $G$ with arities. On the other hand we can apply the nerve theorem to  $(\Theta_G,j_G)$ by Theorem \ref{thm:A-arities-for-G}, and doing so recovers the symmetric simplicial characterisation of groupoids because of
\begin{prp}\label{prop:sym-simp-char}
The inclusion $(\widetilde{\Theta}_1,\tilde{j}_1)\inc(\Theta_G,j_G)$ is a theory equivalence.
\end{prp}
\begin{proof}
The inclusion is full by definition. Let $p$ be a finite connected acyclic graph regarded as an object of $\InvGraph$. Then for any pair of vertices of $p$ there is a unique reduced path between them: existence follows from connectedness, and uniqueness from acyclicity. Thus $Gp$ is the chaotic category on its set of vertices. Thus for some $n \in \N$ one has $Gp \iso Gk_![n]$ and so the inclusion is essentially surjective on objects.
\end{proof}
\begin{rmk}The category inclusion $k:\Delta=\Theta_1\inc\widetilde{\Theta}_1=\Delta_\sym$ is compatible with the theory structures on both sides insofar as $k$ commutes with the arity-inclusion functors $j_\Delta:\Delta_0\inc\Delta$ and $j_{\Delta_\sym}:Seq\inc\Delta_\sym$, where $\Delta_0$ sits in an evident way in $Seq$. We have seen that $(\Delta,j_\Delta)$ is the homogenous theory associated to the free category monad $D_1$ on directed graphs, and that $j_{\Delta_\sym}$ is in a similar way associated to the free groupoid monad $G$ on involutive graphs. It is therefore natural to ask whether there subsists some form of generic/free factorisation system in $\Delta_\sym$.

In the simplex category $\Delta$ the generic morphisms are precisely the endpoint-preserving simplicial operators, cf. \cite[1.13]{Be}. If, accordingly, the generic morphisms in $\Delta_\sym$ are defined to be those which are endpoint-preserving and \emph{either} order-preserving \emph{or} order-reversing then any morphism in $\Delta_\sym$ factors in an essentially unique way as a generic followed by a free morphism; notice however that these generic morphisms do not compose, i.e. condition \ref{dfn:homogeneous}(iii) of a homogeneous theory is not satisfied for $\Delta_\sym$. The reason for this is simple: although $G$-generic factorisations exist by Lemma \ref{lem:G-genfact}, the induced composition of $G$-generic morphisms in the Kleisli category does not necessarily yields $G$-generic morphisms because of possible redundancies.

Nevertheless, the restriction functor $k^*:\widehat{\Delta_\sym}\to\widehat{\Delta}$ extends (under the respective nerve functors) the inclusion of the category of small groupoids into the category of small categories. In particular, left and right Kan extension along $k$ correspond to the well known reflection and coreflection of categories into groupoids.\end{rmk}
\vspace{30ex}

\noindent{\small\sc Universit\'e de Nice-Sophia Antipolis, Lab. J.-A. Dieudonn\'e, Parc Valrose, 06108 Nice, France.}\hspace{2em}\emph{E-mail:}
cberger$@$math.unice.fr\vspace{1ex}

\noindent{\small\sc Universit\'e Paris 7, Laboratoire PPS, Case 7014,
75205 Paris Cedex 13, France.}\hspace{2em}\emph{E-mail:} Paul-Andre.Mellies$@$pps.jussieu.fr\vspace{1ex}

\noindent{\small\sc Department of Mathematics, Faculty of Science, Macquarie University, NSW 2109, Australia.}\hspace{2em}\emph{E-mail:} mark.weber@mq.edu.au

\end{document}